\documentclass[12pt,a4paper,english,reqno]{amsart}
\usepackage[a4paper,footskip=1.5em]{geometry}
\usepackage{amsmath,amssymb,amsthm,mathtools,bbm}
\usepackage[mathscr]{euscript}
\usepackage[usenames,dvipsnames]{color}
\usepackage{adjustbox,tikz,calc,graphics,babel,standalone}
\usepackage{subcaption}
\usepackage{csquotes,enumerate,verbatim}
\usepackage[final]{microtype}
\usepackage[numbers]{natbib}
\usetikzlibrary{shapes.misc,calc,intersections,patterns,decorations.pathreplacing}
\usepackage{hyperref}
\hypersetup{colorlinks=true,linkcolor=blue,citecolor=blue,pdfpagemode=UseNone,pdfstartview={XYZ null null 1.00}}
\usepackage{cmtiup}

\theoremstyle{plain}
\newtheorem*{theorem*}{Theorem}
\newtheorem{theorem}{Theorem}[section]
\newtheorem{lemma}[theorem]{Lemma}
\newtheorem{claim}[theorem]{Claim}
\newtheorem{proposition}[theorem]{Proposition}
\newtheorem*{claim*}{Claim}
\newtheorem{corollary}[theorem]{Corollary}

\newtheorem{problem}[theorem]{Problem}

\theoremstyle{remark}

\def\N{\mathbb{N}}
\def\Z{\mathbb{Z}}

\def\R{\mathbb{R}}
\def\P{\mathbb{P}}
\def\E{\mathbb{E}}

\def\CR{\mathcal{R}}
\def\CB{\mathcal{B}}
\def\B{\mathbf}
\def\Scr{\mathscr}

\DeclareMathOperator\Po{Po}
\DeclareMathOperator\Geom{Geom}
\DeclareMathOperator\Exp{Exp}
\DeclareMathOperator\V{Var}

\let\eps\varepsilon
\newcommand\one{\mathbbm{1}}

\let\originalleft\left
\let\originalright\right
\renewcommand{\left}{\mathopen{}\mathclose\bgroup\originalleft}
\renewcommand{\right}{\aftergroup\egroup\originalright}
\makeatletter
\def\imod#1{\allowbreak\mkern10mu({\operator@font mod}\,\,#1)}
\makeatother

\begin{document}

\title{Coalescence on the real line}

\author{Paul Balister}
\address{Department of Mathematical Sciences, University of Memphis, Memphis TN 38152, USA}
\email{pbalistr@memphis.edu}

\author{B\'{e}la Bollob\'{a}s}
\address{Department of Pure Mathematics and Mathematical Statistics, University of Cambridge, Wilberforce Road, Cambridge CB3\thinspace0WB, UK, \emph{and\/}
Department of Mathematical Sciences, University of Memphis, Memphis TN 38152, USA, \emph{and\/} London Institute for Mathematical Sciences, 35a South St., Mayfair, London W1K\thinspace2XF, UK}
\email{b.bollobas@dpmms.cam.ac.uk}

\author{Jonathan Lee}
\address{Mathematical Institute, University of Oxford, Andrew Wiles Building, Radcliffe Observatory Quarter, Woodstock Road, Oxford OX2\thinspace6GG, UK}
\email{jonathan.lee@merton.ox.ac.uk}

\author{Bhargav Narayanan}
\address{Department of Pure Mathematics and Mathematical Statistics, University of Cambridge, Wilberforce Road, Cambridge CB3\thinspace0WB, UK}
\email{b.p.narayanan@dpmms.cam.ac.uk}

\date{1 October 2016}

\begin{abstract}
We study a geometrically constrained coalescence model derived from spin systems. Given two probability distributions $\P_R$ and $\P_B$ on the positive reals with finite means, colour the real line alternately with red and blue intervals so that the lengths of the red intervals have distribution $\P_R$, the lengths of the blue intervals have distribution $\P_B$, and distinct intervals have independent lengths. Now, iteratively update this colouring of the line by coalescing intervals: change the colour of any interval that is surrounded by longer intervals so that these three consecutive intervals subsequently form a single monochromatic interval. We say that a colour (either red or blue) \emph{wins} if every point of the line is eventually of that colour. Holroyd, in 2010, asked the following question: under what natural conditions on the initial distributions is one of the colours almost surely guaranteed to win? It turns out that the answer to this question can be quite counter-intuitive due to the non-monotone dynamics of the model. In this paper, we investigate various notions of `advantage' one of the colours might initially possess, and in the course of doing so, we determine which of the two colours emerges victorious for various nontrivial pairs of initial distributions.
\end{abstract}

\maketitle

\section{Introduction}
The object of study in this paper is a one-dimensional geometrically constrained coalescence model. This model, \emph{coalescence on the real line}, describes the evolution of a colouring of the real line into intervals: a colouring $\Delta$ of the real line $\R$ with two colours, red and blue, is a \emph{colouring into intervals} if there is a doubly infinite sequence of points $(p_i)_{i \in \Z}$, with $p_i < p_j$ when $i < j$ and $\R = \bigcup_{i \in \Z} (p_i, p_{i+1}]$, such that the interval $(p_{2i-1}, p_{2i}]$ is coloured red and the interval $(p_{2i}, p_{2i+1}]$ is coloured blue for each $i \in \Z$; we call the points $(p_i)_{i \in \Z}$ the \emph{boundary-points} of the colouring. In coalescence on the real line, or \emph{linear coalescence} for short, we evolve a colouring of the real line into intervals by repeatedly coalescing intervals together according to the following rule: change the colour of any monochromatic interval of the colouring that is surrounded by longer monochromatic intervals of the opposite colour so that the three consecutive intervals are of the same colour; these three intervals are subsequently taken to be a single monochromatic interval. We call such a step in which three consecutive monochromatic intervals of the colouring are merged into a single monochromatic interval a \emph{recolouring}; see Figure~\ref{amalg} for an illustration. Here, our primary focus is the following question: what can we say about the dynamics of linear coalescence when the initial colouring is stochastic? To make this question precise, we need a few definitions.

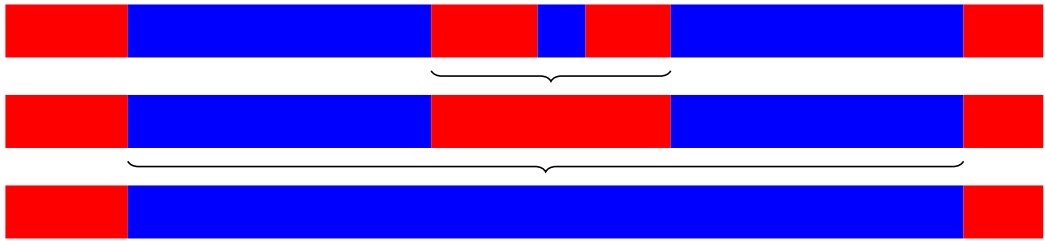
\begin{figure}
	\begin{center}
		\begin{tikzpicture}[xscale = 0.7,yscale = 0.8]
		\draw [red, line width=20](-8,4)--(-5.7,4);
		\draw [blue, line width=20](-5.7,4)--(0,4);
		\draw [red, line width=20](0,4)--(2,4);
		\draw [blue, line width=20](2,4)--(2.9,4);
		\draw [red, line width=20](2.9,4)--(4.5,4);
		\draw [blue, line width=20](4.5,4)--(10,4);
		\draw [red, line width=20](10,4)--(11.5,4);
		
		\draw [red, line width=20](-8,2.5)--(-5.7,2.5);		
		\draw [blue, line width=20](-5.7,2.5)--(0,2.5);
		\draw [red, line width=20](0,2.5)--(4.5,2.5);
		\draw [blue, line width=20](4.5,2.5)--(10,2.5);
		\draw [red, line width=20](10,2.5)--(11.5,2.5);
		
		\draw [red, line width=20](-8,1)--(-5.7,1);
		\draw [blue, line width=20](-5.7,1)--(10,1);
		\draw [red, line width=20](10,1)--(11.5,1);
		
		\draw [ultra thick, decorate, decoration={brace,amplitude=4pt,mirror}, semithick, yshift=4pt]
		(0,3.2) -- (4.5,3.2);
		
		\draw [ultra thick, decorate, decoration={brace,amplitude=4pt,mirror}, semithick, yshift=4pt]
		(-5.7,1.7) -- (10,1.7);
		\end{tikzpicture}
	\end{center}
	\caption{A sequence of two recolourings.}
	\label{amalg}
\end{figure}

Given a pair of probability distributions $\P_R$ and $\P_B$ on the positive reals with finite means $\mu_R$ and $\mu_B$ respectively, we may construct a colouring of the real line into intervals so that the lengths of the red intervals have distribution $\P_R$, the lengths of the blue intervals have distribution $\P_B$, and distinct intervals have independent lengths. Indeed, let $(\CR_i)_{i \in \Z}$ and $ (\CB_i)_{i \in \Z}$ be two  i.i.d.\ sequences of random variables (which are additionally independent of each other) with distributions $\P_R$ and $\P_B$ respectively; we then write $\Delta(\P_R, \P_B)$ for the random colouring of the real line into intervals constructed as follows: colour the interval $(-\CR_0, 0]$ red and the interval $(0, \CB_0]$ blue, then colour the interval $(\CB_0, \CR_1 + \CB_0]$ red and the interval $(-\CR_0 - \CB_{-1}, -\CR_0]$ blue, and so on, adding intervals of alternating colours to the left and right inductively. Clearly, the above colouring is shift-invariant with respect to the underlying sequences of interval lengths. Let us also note that there is nothing special about the precise choice of origin in the above construction, and also that the sequence of boundary-points of a colouring generated as described above is unbounded in both directions almost surely.

We say that a colouring $\Delta$ of the real line into intervals with boundary-points $(p_i)_{i \in \Z}$ is \emph{non-degenerate} if $p_k - p_j \neq p_l - p_k$ for any three boundary-points $p_j$, $p_k$ and $p_l$ with $j < k < l$. It is clear that no two adjacent monochromatic intervals will ever have the same length when we iteratively recolour intervals starting from a non-degenerate colouring. Henceforth, to avoid unnecessary complications, all our colourings of the line into intervals will be assumed to be non-degenerate. Observe that if at least one of $\P_R$ or $\P_B$ is non-atomic, then $\Delta(\P_R, \P_B)$ is almost surely non-degenerate; therefore, when considering stochastic colourings of the form $\Delta(\P_R, \P_B)$, we shall assume implicitly that at least one of $\P_R$ and $\P_B$ is non-atomic.

A sequence of recolourings is said to be \emph{complete} (with respect to an initial colouring $\Delta$) if whenever there exists a monochromatic interval $I$ surrounded by longer monochromatic intervals of the opposite colour at some stage of the evolution (in linear coalescence starting from $\Delta$), then $I$ is eventually recoloured. To see that the notion of a complete sequence of recolourings is meaningful, note that if a monochromatic interval $I$ can be recoloured, then its neighbours cannot; as its neighbours can only grow longer, recolouring $I$ remains an option forever. 

Here, we shall aim to understand the dynamics of linear coalescence starting from a stochastic initial colouring; the precise question that we shall be concerned with is the following.

\begin{problem}\label{p:first}
Given $\P_R$ and $\P_B$ at least one of which is non-atomic, what is the result of a complete sequence of recolourings applied to $\Delta(\P_R, \P_B)$?
\end{problem}

We shall see that for any such pair of distributions $\P_R$ and $\P_B$, there are only three possible outcomes (each of which has probability either $0$ or $1$): either every point changes colour finitely many times and is eventually red, every point changes colour finitely many times and is eventually blue, or the colour of every point changes infinitely often. Furthermore, it turns out to be the case (see Proposition~\ref{p:unique}) that the outcome depends only on the distributions themselves and is independent of the choice of complete sequence of recolourings; on account of this fact, when we speak about the `dynamics' of linear coalescence in the sequel, we shall mean the long-term behaviour of the evolution under an \emph{arbitrary} complete sequence of recolourings. 

Given a pair of distributions $\P_R$ and $\P_B$, we say that a colour (either red or blue) \emph{wins} if every point is eventually of that colour under any complete sequence of recolourings applied to $\Delta(\P_R, \P_B)$; on the other hand, if the colour of every point changes infinitely often in linear coalescence starting from  $\Delta(\P_R, \P_B)$, then we say that it is a \emph{tie}. 

It is easy to see that if $\P_R = \P_B$, then we must have a tie by symmetry. However, when $\P_R \neq \P_B$, one would expect the coalescence process to amplify any `advantage' possessed initially by one of the colours. We therefore restrict our attention to the following question in this paper.

\begin{problem}\label{p:main}
Under what natural conditions on the distributions $\P_R$ and $\P_B$ is one of the colours guaranteed to win?
\end{problem}

The coalescence model considered in this paper (and Problem~\ref{p:main} in particular) was proposed by Holroyd~\citep{lunch}. The primary motivation for studying this model comes from trying to better understand the behaviour of various models for interacting particle systems in the statistical physics literature. Many spin particle models~\citep{motiv1, motiv2, motiv3, motiv4} have been introduced in the physics literature to model, amongst other things, the liquid-glass transition, the formation of domains in magnetic systems, and the evolution of liquid droplets. In such models, one usually has a particle at each point of the cubic lattice $\Z^d$ and one specifies the state of each such particle; the model comes equipped with a local rule which governs the evolution of the states of the particles and one is typically interested in understanding the dynamics of the evolution of a random initial configuration. Coalescence on the line can be seen (when we discretise the real line, place a particle at each such resulting point, and then identify the colours red and blue with the two possible spin states of these particles) to be an example of such a spin particle model on $\Z$ where the local update rule goes beyond merely considering nearest-neighbour interactions. When the discretisation is taken to be fine, our model can in fact be seen exactly as the zero-temperature and large-domain limit of the coarsening of magnetic domains in one dimension; see~\citep{BRAY}, for example.

The process we study can also be thought of as a `hierarchical coalescence process' where the coalescences involving short intervals occur first. A variant of this process with $\P_R = \P_B$ where the recolourings take place in a sequence of distinct epochs has been the subject of some recent work; interesting results about such hierarchical coalescence processes have been obtained by Faggionato, Martinelli, Roberto and Toninelli~\citep{hcp1, hcp2}. Results about the dynamics of such hierarchical coalescence processes have been useful in explaining the universality in the limiting behaviour of spin particle models like the East model~\citep{East}; for details, see~\citep{hcp3}.

Problems~\ref{p:first} and~\ref{p:main} are also closely related to the question of constructing `stable matching schemes', a well-studied problem in combinatorics~\citep{stab_comb1, stab_comb2}, probability~\citep{stab_prob1, stab_prob2} and statistical physics~\citep{stab_phys1, stab_phys2}. Given a colouring of the real line into intervals, we can construct a perfect matching between the red and blue boundary-points of this colouring as follows: we coalesce the intervals using some complete sequence of recolourings, and every time we recolour an interval, we match the endpoints of that interval to each other. The resulting perfect matching between the red and blue boundary-points, assuming that each boundary-point prefers to be matched to a boundary-point of the opposite colour as close to it as possible, is easily seen to be stable in the sense of Gale and Shapley~\citep{GaleShapley}. For related work on constructing matchings on point sets arising from various point processes in a Euclidean space, we refer the reader to the papers of Ajtai, Koml{\'o}s and Tusn{\'a}dy~\citep{euclid1}, Holroyd, Pemantle, Peres and Schramm~\citep{euclid2}, and Holroyd~\citep{euclid3}.

Let us remark briefly that if we only coalesce intervals of a single colour, say blue, where adjacent blue intervals of length $x$ and $y$ merge together at rate $K(x,y)$ when both intervals are longer than the red interval between them, then we recover a geometrically constrained variant of the Marcus--Lushnikov model for stochastic coalescence; we refer the reader to the survey of Aldous~\citep{Aldous} for more about the Marcus--Lushnikov model and its variants. The main difference between such models and the model considered in this paper, and perhaps what makes linear coalescence particularly interesting, concerns monotonicity. Indeed, if we only merge blue intervals together, then it is not hard to see that the process is monotone with respect to the blue distribution, assuming of course that the rate kernel $K$ is suitably monotone. However, since a point can be recoloured an arbitrarily large number of times in linear coalescence, it turns out (see Claim~\ref{t:counter}) that we cannot expect any such monotonicity. Indeed, it would appear (see Claim~\ref{t:trans}) that the relationship between initial distributions induced by `winning' in the sense of Problem \ref{p:main} is \emph{not even transitive}!

It is also possible to study linear coalescence in continuous time. Indeed, place a balloon at each of the boundary-points and inflate these balloons at rate $1/2$ so that a balloon centred at some boundary-point has radius $t/2$ at time $t$; when two balloons meet, they both pop and we match the centres of these balloons and remove them. It is not hard to check that if we look at the balloons which remain at some time $t>0$, they correspond precisely to the endpoints of intervals in the process once every interval of length less than $t$ has been iteratively removed. One benefit of studying the process in continuous time is that one can say a great deal about the process in the case where $\P_R = \P_B$. As we remarked earlier, when $\P_R = \P_B$, the outcome of linear coalescence is a tie. However, in this case, one can actually use the machinery of hierarchical coalescence processes to say a lot more about the normalised lengths of the surviving intervals at each time $t>0$. By observing that the sequence of boundary-points which remain at any time $t$ is a renewal process, Eccles and Holroyd~\citep{Holroyd} have obtained results about the probability that a boundary-point survives to time $t$ in this continuous time process for various initial distributions, showing, for example, that in many cases, this probability is asymptotic to $K/t$ as $t \to \infty$ for some explicit constant $K>0$.

However, the techniques discussed above seem to be of little use when $\P_R \neq \P_B$. It is reasonable to believe that if one of the colours has enough of an advantage to begin with, then this advantage should amplify and that colour should win; the main difficulty appears to lie in finding the right notion of advantage, however. It is clear that blue wins when, for some $L > 0$, the length of each initial blue interval exceeds $L$ and the length of each initial red interval is at most $L$; indeed, we can recolour all the red intervals in one step. However, the task of proving that a particular colour wins for any \emph{nontrivial} pair of distributions does not seem to be straightforward. In this paper, we shall develop some combinatorial techniques to track the dynamics of linear coalescence in the case where $\P_R \neq \P_B$, and then use these techniques to investigate various natural notions of advantage; in doing so, we shall decide the outcome of linear coalescence for various nontrivial pairs of distributions.

\section{Our results}
Before we state our results, we remind the reader that here, and in what follows, we shall restrict our attention to the evolution of a non-degenerate initial colouring of the real line into intervals under a complete sequence of recolourings; when considering stochastic colourings of the form $\Delta(\P_R, \P_B)$ in particular, we shall assume implicitly that at least one of $\P_R$ and $\P_B$ is non-atomic. We begin with the following proposition that establishes the appropriate setting for our results.

\begin{proposition}\label{p:unique}
For any pair of probability distributions $\P_R$ and $\P_B$ (at least one of which is non-atomic), under any complete sequence of recolourings applied to $\Delta(\P_R, \P_B)$, either 
\begin{enumerate}
\item all points are eventually red almost surely (a red-win), or 
\item all points are eventually blue almost surely (a blue-win), or
\item every point changes colour infinitely often almost surely (a tie);
\end{enumerate}
furthermore, the outcome is independent of the choice of complete sequence of recolourings.
\end{proposition}

In the light of Proposition~\ref{p:unique}, it is natural to expect that if the initial distribution of one of the colours has enough of an advantage over the initial distribution of the other colour, then this advantage should amplify and that colour should win; here, we shall consider two natural notions of advantage in linear coalescence. 

The first, and perhaps most elementary, notion of advantage that one can consider is based simply on a first moment condition: does a colour win almost surely if the mean of its distribution is substantially bigger than that of the distribution of the other colour? Our first result shows that this is not the case.

\begin{theorem}\label{t:mean}
For any $K>0$, there exist distributions $\P_R$ and $\P_B$ with $\mu_R>K\mu_B$ for which blue wins almost surely.
\end{theorem}

The next, and much stronger, notion of advantage that we consider is that of stochastic dominance. We say that \emph{$\P_R$ stochastically dominates $\P_B$}, and write $\P_R \succcurlyeq \P_B$, if $\P_R([x, \infty)) \ge \P_B([x, \infty))$ for every $x \ge 0$; if this inequality is strict for at least one $x > 0$, we write $\P_R \succ \P_B$ and say that \emph{$\P_R$ stochastically dominates $\P_B$ strictly}. It is tempting to conjecture that a colour wins almost surely if its distribution stochastically dominates the distribution of the other colour strictly; however, rather counter-intuitively, this does not appear to be the case. In fact, even a combination of stochastic dominance and the first moment condition considered above is not sufficient to guarantee victory.

\begin{claim}\label{t:counter}
With very high confidence, for any $K>0$, there exist distributions $\P_R$ and $\P_B$ such that $\P_R \succ \P_B$ and $\mu_R>K\mu_B$ for which blue wins almost surely.
\end{claim} 

Our next result provides further evidence that linear coalescence is far from being monotone with respect to the lengths of the intervals in the initial colouring. Consider the relation $\rhd$ on the space of probability distributions on the positive reals with finite means defined by saying that $\P_R \rhd \P_B$ if red wins when the initial red and blue distributions are $\P_R$ and $\P_B$ respectively. One of the main difficulties in analysing linear coalescence stems from the fact that $\rhd$ does not appear to be transitive.

\begin{claim}\label{t:trans}
With very high confidence, there exist distributions $\P_R$, $\P_G$ and $\P_B$ such that $\P_R \rhd \P_G$, $\P_G \rhd \P_B$ and $\P_B \rhd \P_R$.
\end{claim} 

The phrase `with very high confidence' merits explanation. We show that the proofs of Claims~\ref{t:counter} and~\ref{t:trans} may be reduced to the task of bounding certain \emph{finite-dimensional} numerical integrals, which are nevertheless of sufficiently high dimension as to be impractical to evaluate. Instead, we \emph{estimate} them by Monte Carlo methods. Furthermore, we show that if it was the case that the required bounds did not hold, the probability of our Monte Carlo results would be very small (of the order $10^{-12}$). In the past, such Monte Carlo methods have been used to prove high confidence intervals for the critical probabilities of various percolation models; for some background, we refer the reader to the papers of Bollob\'as and Stacey~\citep{highconf1}, and Balister, Bollob\'as and Walters~\citep{highconf2}.

When we are given a pair of distributions $\P_R$ and $\P_B$ and are faced with the task of deciding the outcome of linear coalescence with these initial distributions, a natural first step is to study the coalescence process with these distributions on a large finite interval. If a large fraction of this finite interval turns blue (as in Figure~\ref{finite}) with a reasonably large probability, one would be inclined to believe that blue wins. It is possible in certain cases (see Theorem~\ref{t:renorm}) to actually deduce the outcome of linear coalescence from the typical outcome of the coalescence process on a large finite interval; the proofs of Theorem~\ref{t:mean}, Claim~\ref{t:counter} and Claim~\ref{t:trans} hinge upon this idea.

While stochastic dominance, in the light of Claim~\ref{t:counter}, does not seem to be enough in general to guarantee victory, our next theorem, which is a positive result, states that if one of the colours has a `sufficiently large' stochastic advantage to begin with, then that colour wins. To state this theorem, our main result, we need a few definitions. Define $\Scr{F}(\lambda)$ to be the \emph{shifted exponential distribution} on $[1,\infty)$ with density function
\[ 
\frac{\exp(-(x-1)/\lambda)}{\lambda} \one_{\{x\ge 1\}}.
\]
In other words, if the distribution of a random variable $X$ is $\Scr{F}(\lambda)$, then  $X-1$ is an exponential random variable with mean $\lambda$. Next, define $\Scr{G}(a)$ to be the \emph{Pareto distribution} on $[1,\infty)$ with density function
\[
\frac{2(a+1)^2}{(a+x)^3}\one_{\{x\ge 1\}}.
\]
Note that if the distribution of a random variable $X$ is $\Scr{G}(a)$, then $\P(X\ge x)=(a+1)^2/(a+x)^2$ for all $x\ge1$. With these definitions in place, we can now state our main theorem.

\begin{theorem}\label{t:main}
There exists a constant $\Lambda < 14$ such that for all $\lambda > \Lambda$ and all $a \in [0,1)$, if $\P_B \succcurlyeq \Scr{F}(\lambda)$ and $\Scr{G}(a) \succcurlyeq \P_R$, then almost surely blue wins.
\end{theorem}

\begin{figure}
	\begin{subfigure}{\textwidth}
		\centering
		\includegraphics[width= 0.95\textwidth, height = 4.8cm]{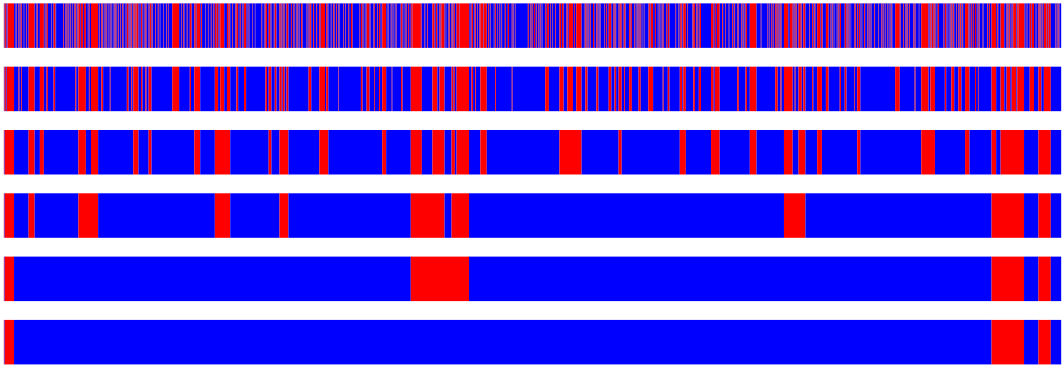}
		\vspace{0.5cm}
	\end{subfigure}
\caption{The coalescence process with $\P_R = \Scr{G}(1)$ and $\P_B = \Scr{F}(3)$ restricted to a finite interval.}
\label{finite}
\end{figure}

It is worth noting that the distributions $\Scr{F}(\lambda)$ and $\Scr{G}(a)$ are stochastically incomparable, i.e., $\Scr{F}(\lambda) \not\succcurlyeq \Scr{G}(a)$ and $\Scr{G}(a) \not\succcurlyeq \Scr{F}(\lambda)$ for all $\lambda > \Lambda$ and $0\le a<1$. The choice of the blue distribution $\Scr{F}(\lambda)$ is fairly natural as we shall see that once we have eliminated very short intervals, the distribution of the lengths of the remaining blue intervals always dominates some exponential distribution. The choice of the red distribution $\Scr{G}(a)$ appears to be a bit more artificial; however, the proof naturally asks for (and applications naturally want) a distribution with a fat tail, and this seems to be the simplest distribution for which our proof works. As an application of Theorem~\ref{t:main}, we establish the following result. 

\begin{theorem}\label{t:toy}
Suppose that all the red intervals are initially of length $1$, and that the initial lengths of the blue intervals are uniformly distributed on the interval $[0,1+\gamma]$. There exist positive constants $\gamma_R = 0.1216$ and $\gamma_B = 6.048$ such that if $0 \le \gamma \le \gamma_R$, then red wins almost surely, and if $\gamma \ge \gamma_B$, then blue wins almost surely.
\end{theorem}

Simulations suggest, as one might expect, that there is a phase transition in the context of Theorem~\ref{t:toy} from a red-win to a blue-win at some critical value $\gamma_c \in [1.16, 1.19]$; however, due to the non-monotone nature of linear coalescence, we are unable to even rule out the possibility that there exist infinitely many values of $\gamma$ in the interval $[\gamma_R, \gamma_B]$ at which the outcome flips!

The rest of this paper is organised as follows. In Section~\ref{s:prelim}, we introduce some notation and prove Proposition~\ref{p:unique}. In Section~\ref{s:bound}, we describe a method which can, in certain cases, be used to deduce the outcome of linear coalescence from the typical outcome of the process on a large finite interval; we then prove Theorem~\ref{t:mean} as well as our high confidence results, Claims~\ref{t:counter} and~\ref{t:trans}, in Section~\ref{s:renorm}. To prove our main result, we develop in Section~\ref{s:strat} a method for tracking the evolution by maintaining a collection of `approximate interval lengths' that is, crucially, stable under stochastic dominance. We then prove Theorem~\ref{t:main} in Section~\ref{s:main} and describe how to deduce Theorem~\ref{t:toy} from Theorem~\ref{t:main} in Section~\ref{s:application}. Finally, we conclude the paper in Section~\ref{s:conc} with a discussion of some open problems. 

The computer programs used to establish the aforementioned high confidence results, and to rigorously verify certain other estimates, are available for download at
\href{https://www.dpmms.cam.ac.uk/~bp338/code/coalescence.zip}{{\ttfamily www.dpmms.cam.ac.uk/$\sim$bp338/code/coalescence.zip}}.

\section{Preliminaries}\label{s:prelim}
In this section, we shall prove a few elementary results about linear coalescence in the case in which we have only finitely many intervals, justify our claims about the outcome of linear coalescence, and finally, establish some notational conventions for the rest of the paper.

By a \emph{coloured-interval}, we mean an interval consisting of a finite number of subintervals coloured alternately red and blue. We shall consider the coalescence process on a coloured-interval; all our terminology in the case of linear coalescence carries over to this setting as well. In what follows, all coloured-intervals that we consider will be \emph{non-degenerate}.

It is clear that any complete sequence of recolourings for linear coalescence on a coloured-interval is of finite length. The following simple fact is, perhaps, not so immediate.

\begin{lemma}\label{l:order}
The number of times a point of a (non-degenerate) coloured-interval changes colour in a complete sequence of recolourings is independent of the complete sequence. In particular, the final state of linear coalescence on a coloured-interval is independent of the order of recolourings.
\end{lemma}
\begin{proof}
We claim that the number of recolourings of a point in any complete sequence of recolourings is the same as in the sequence of recolourings where one inductively recolours the \emph{shortest recolourable interval}, i.e., the shortest monochromatic interval that is surrounded by longer monochromatic intervals on both sides, at each step.

If there is no recolourable interval to begin with, then there is nothing to prove. So suppose that $R$ is the shortest (say red) recolourable interval in the original configuration, and that it is surrounded by the longer (blue) intervals $B_-$ and $B_+$. Then in any complete sequence of recolourings, it is impossible to recolour either $B_-$ or $B_+$ before recolouring $R$. Thus, $R$ is always surrounded by longer intervals up until the time it is (necessarily) recoloured in the given complete sequence, say at step $t$. Compare this with the sequence where we recolour $R$ first, but otherwise keep the order of recolourings the same. Any red interval that gets recoloured and was surrounded by an interval containing $B_-$ or $B_+$ in the given complete sequence is now surrounded by an even longer blue interval containing $B_-$,  $R$ and $B_+$, so the recolouring is still valid and still recolours exactly the same set of points. Finally after step $t$ in the new sequence of recolourings, recolouring $R$ is unnecessary and we find that we are in exactly the same state as in the given complete sequence after step $t$, with each point recoloured the same number of times. Hence, we may assume that $R$ is recoloured first. Applying induction on the number of intervals, we see that the order of recolourings from the next step onwards is irrelevant when we recolour $R$ first. Thus, we may as well always recolour in the order of shortest recolourable interval first. As this ordering is uniquely defined (for any non-degenerate coloured-interval), it follows that the number of times each point is recoloured is independent of the complete sequence of recolourings. 

The final state of the coalescence process on the coloured-interval depends only on the number of times each point is recoloured, so it is also independent of the order of recolourings.
\end{proof}

Lemma~\ref{l:order} allows us to uniquely define the final state $[C]$ of linear coalescence on a coloured-interval $C$. We call $[C]$ the \emph{closure} of $C$, and say that $C$ is \emph{closed} if $C=[C]$. Note that $C$ is closed if and only if the lengths of its monochromatic subintervals form a sequence with no local minimum, i.e., if this sequence of lengths consists of a (possibly trivial) increasing subsequence followed by a (possibly trivial) decreasing subsequence.

Denote by $C_1+C_2$ the concatenation of the coloured-intervals $C_1$ and $C_2$; if $C_1$ ends and $C_2$ starts with intervals of the same colour then we merge these into a single monochromatic interval in $C_1 + C_2$. As a consequence of Lemma~\ref{l:order}, we see that $[[C_1]+[C_2]]=[C_1+C_2]$ for any two coloured-intervals $C_1$ and $C_2$.

For $x\in C$, write $N_x(C)$ for the number of times the point $x$ is recoloured in any complete sequence of recolourings of the coloured-interval $C$. We first note that this is well-defined in the light of Lemma~\ref{l:order}. Moreover, we note that 
\begin{equation}\label{eq:mono}
N_x(C_-+C+C_+)\ge N_x(C)
\end{equation}
for any pair of coloured-intervals $C_-$ and $C_+$; indeed, in the coalescence process on $C_-+C+C_+$, we can always start by performing a complete sequence of recolourings inside $C$ first.

Consider a colouring $\Delta$ of the real line into monochromatic intervals given by
\[\dots +R_{-1}+B_{-1}+R_0+B_0+R_1+B_1+\dots,\]
where $R_i$ is a red interval and $B_i$ is a blue interval for each $i\in \Z$. It follows from~\eqref{eq:mono} that the number $N_x(R_{-n}+B_{-n}+\dots+B_n)$ of times a point $x\in\R$ is recoloured in the coalescence process restricted to $R_{-n}+B_{-n}+\dots+B_n$ is monotone in $n$, so we can define its limit $N_x(\Delta) \in\N\cup\{\infty\}$.

We now establish Proposition~\ref{p:unique} by proving the following lemma that establishes a little bit more than the proposition; recall that a sequence of recolourings is complete if whenever there is a monochromatic interval at some time $t$ that can be recoloured, then it is recoloured at some time $T \ge t$.

\begin{lemma}\label{l:unique}
Let $\Delta$ be any (non-degenerate) colouring of the real line into intervals. Under any complete sequence of recolourings applied to $\Delta$, each point $x \in \R$ is recoloured exactly $N_x = N_x(\Delta)$ times. Also, if $\Delta = \Delta(\P_R, \P_B)$ for some pair of probability distributions $\P_R$ and $\P_B$ (at least one of which is non-atomic), then under any complete sequence of recolourings applied to $\Delta$, either 
	\begin{enumerate}
	\item all points are eventually red almost surely, or 
	\item all points are eventually blue almost surely, or
	\item $N_x=\infty$ for all $x\in\R$ almost surely;
\end{enumerate}
furthermore, the outcome is independent of the choice of complete sequence of recolourings.
\end{lemma}
\begin{proof}
We first show that each point $x \in \R$ is recoloured exactly $N_x$ times in any complete sequence of recolourings. Note that the number of times a point $x \in \R$ can be recoloured in any finite sequence of recolourings is at most $N_x$: this is obviously true if $N_x = \infty$; if $N_x < \infty$ and $x$ is recoloured $N_x+1$ times in some finite sequence of recolourings, then such a recolouring would be valid when restricted to some coloured-interval $C = R_{-n}+\dots+B_n$, contradicting the fact that $N_x(R_{-n}+\dots+B_n)\le N_x$. Furthermore, there is a complete sequence of recolourings which does recolour every point exactly $N_x$ times; indeed, consider a complete sequence which recolours $R_{-1}+\dots+B_{1}$, then extend this to a complete sequence recolouring $R_{-2}+\dots+B_{2}$, and so on. 

Now, suppose that a given complete sequence recolours intervals in the order $I_1, I_2,\dots$, and that the above sequence recolours intervals in the order $I^*_1, I^*_2,\dots$; we shall show that for any $t$, there is an $n_t$ such that it is possible to change, without altering the number of times any point is recoloured, the sequence of recolourings $I_1,\dots,I_{n_t}$ to a new sequence of recolourings $I'_1,\dots,I'_{n_t}$ with $I'_i= I^*_i$ for $i\le t$. Indeed, by induction it is enough to do this for $t=1$. As $I^*_1$ is surrounded by longer intervals to begin with, this interval must eventually appear as some $I_{n_1}$. As in the proof of Lemma~\ref{l:order}, we may just define $I'_1= I^*_1$ and $I'_{i+1}=I_i$ for $i<n_1$.

Next, given $x\in\R$, for each natural number $n \le N_x$, we can choose a sufficiently large $t$ so that $x$ is recoloured $n$ times via the sequence $I^*_1,\dots, I^*_t$. Then $x$ is recoloured at least $n$ times by $I_1,\dots,I_{n_t}$. As $x$ is eventually recoloured $N_x$ times by the sequence $I^*_1, I^*_2, \dots$, it must be recoloured at least $N_x$ times by the sequence $I_1, I_2,\dots$, establishing our first claim.

If $N_x=\infty$ for some $x\in R_i\cup B_i$, say, then after $x$ is recoloured $n$ times, each point of $R_j\cup B_j$  is recoloured at least $n-2|i-j|$ times; indeed, if $n \ge 2|i-j|$, then the points in $R_j\cup B_j$ must lie in the same monochromatic interval as the points in $R_i \cup B_i$ after $2|i-j|$ recolourings of the points in $R_i \cup B_i$. Consequently, either $N_x = \infty$ for all $x \in \R$, or $N_x < \infty$ for all $x \in \R$. If the latter conclusion holds, then every point of the real line subsequently has a well-defined final colour. Moreover the final colouring partitions $\R$ into a sequence $(J_i)_{i\in S}$ of (finite or infinite) intervals of alternating colours, and as no further recolouring is possible, the lengths of these intervals must form a strictly unimodal sequence. 

It is impossible for the sequence $(J_i)_{i\in S}$ to be finite and consist of more than two intervals as the two end-intervals would have infinite length. We now claim that it is almost surely impossible for this sequence to be infinite in the case where $\Delta = \Delta(\P_R, \P_B)$. Indeed, if the sequence is infinite in one direction only, say $J_0,J_1,\dots$, then as $|J_0|=\infty$, it must be the case that the sequence $(|J_i|)_{i \ge 1}$ is decreasing. Similarly, if $(J_i)_{i\in S}$ is a two-way infinite sequence, then the lengths of these intervals must decrease in at least one direction. Thus, in either case, all the $J_i$ must have a bounded length, say $|J_i|<L$, from some point onwards. This cannot happen in the case where $\Delta = \Delta(\P_R, \P_B)$, because with probability $1$, there is a $k>0$ and a positive density of indices $i$ such that the sequence $R_i + B_i + R_{i+1} + \dots + B_{i+k}$ in the original colouring coalesces into a single interval of length at least $L$. For example, without loss of generality, there is an $L'$ such that with positive probability, an original red interval has length less than $L'$ and an original blue interval has length at least $L'$. A sufficiently long alternating sequence of such intervals will coalesce into a blue interval of length at least $L$, and such configurations will occur with positive density. 

Hence, we conclude from the discussion above that if $\Delta = \Delta(\P_R, \P_B)$, then the sequence $(J_i)_{i\in S}$ either consists of a single monochromatic interval, or an infinite red interval followed by an infinite blue interval, or an infinite blue interval followed by an infinite red interval. Recall that $\Delta(\P_R, \P_B)$ is constructed as follows: we take two  i.i.d.\ sequences $(\CR_i)_{i \in \Z}$ and $(\CB_i)_{i\in\Z}$ of random variables with distributions $\P_R$ and $\P_B$ respectively, and then define $\Delta(\P_R, \P_B)$ to be the colouring of the real line given by
\[\dots +R_{-1}+B_{-1}+R_0+B_0+R_1+B_1+\dots,\]
where $R_i$ is a red interval with $|R_i|= \CR_i$ and $B_i$ is a blue interval with $|B_i| = \CB_i$ for each $i\in \Z$, and the origin is the boundary-point between $R_0$ and $B_0$. Now, the law of $\Delta(\P_R, \P_B)$ is \emph{shift-invariant} under integer-shifts of the form $(\CR_i)_{i \in \Z} \to (\CR_{i+j})_{i \in \Z}$ and $(\CB_i)_{i\in\Z} \to (\CB_{i+j})_{i \in \Z}$  for any $j\in\Z$. Furthermore, the events `$N_x=\infty$ for all $x\in \R$', `every point is eventually red', `every point is eventually blue', `there is a $y\in \R$ such that every $x \le y$ is eventually red and every $x>y$ is eventually blue', and `there is a $y\in \R$ such that every $x \le y$ is eventually blue and every $x>y$ is eventually red' are all shift-invariant events that are measurable with respect to the $\sigma$-algebra generated by finite subsets of  the random variables $(\CR_i)_{i \in \Z}$ and $(\CB_i)_{i\in\Z}$. Hence, by ergodicity, they all occur with probability either $0$ or 1. By symmetry, the latter two events have the same probability which must necessarily be 0. 

As the number of times each point is recoloured is independent of the choice of complete sequence of recolourings, each of the three possible outcomes (namely a red-win, a blue-win, and a tie) is also independent of the choice of complete sequence of recolourings.
\end{proof}

We close this section by establishing a few notational conveniences. For a probability distribution $\Scr{D}$ and a non-negative integer $k \in \Z$, we write $k \circ \Scr{D}$ for the distribution of the sum $\sum_{i=1}^k X_i$, where $X_1, X_2, \dots, X_k$ are independent random variables with distribution $\Scr{D}$; more generally, given a distribution $\Scr{Z}$ on the non-negative integers, we write $\Scr{Z} \circ \Scr{D}$ for the distribution of the sum of $\sum_{i=1}^Z X_i$ where $Z$ is a random variable with distribution $\Scr{Z}$ and $X_1, X_2, \dots, X_Z$ are independent random variables with distribution $\Scr{D}$, all of which are independent of $Z$. For a distribution $\Scr{D}$ and $L\in \R$, we say that $X$ has distribution $L + \Scr{D}$ if the distribution of $X-L$ is $\Scr{D}$. Finally, we shall write 
\begin{enumerate}
\item $\Exp(\lambda)$ for the (exponential) distribution of a non-negative random variable $X$ such that $\P(X \ge x) = \exp(-x/\lambda)$ for each $x \ge 0$, 
\item $\Po(\lambda)$ for the (Poisson) distribution of a random variable $X$ supported on the non-negative integers with $\P (X = k) = \lambda^k e^{-\lambda}/k!$ for each integer $k \ge 0$,
\item $\Geom(p)$ for the (geometric) distribution of a random variable $X$ supported on the positive integers with $\P (X = k) = p^{k-1} (1-p)$ for each integer $k \ge 1$, and
\item $U[a,b]$ for the (uniform) distribution of a random variable distributed uniformly on the interval $[a,b]$.
\end{enumerate}

\section{The renormalisation argument}\label{s:bound}
We now describe a strategy, which we call \emph{the renormalisation argument}, to deduce the outcome of linear coalescence from the typical outcome of the coalescence process on a large finite interval. 

Let $\B{R} \subset \R \times \R \times \N$ denote the set of triples $( \alpha, \beta, k)$ with $0<\alpha \le 1/4$ and $\beta > 1$ such that
\begin{align}
 \beta+\alpha\beta&<2-3\alpha, \mbox{ and}\label{e:b1}\\
 \alpha(k-3)&>2\beta(1-\alpha).\label{e:b2}
\end{align}
The set $\B{R}$ is nonempty; we can check, for example, that $(1/5, 10/9, 12) \in\B{R}$. We call a triple of $\B{R}$ \emph{renormalisable}. First, note that for any renormalisable triple, we have $k > 6$. Also, note that~\eqref{e:b1} is equivalent to
\[
\frac{5}{3+\beta} > 1+\alpha,
\]
so the bound on $\alpha$ falls from $1/4$ to $0$ as $\beta$ goes from $1$ to $2$. For any $\alpha$ and $\beta$ satisfying this inequality, there is always a (large) $k$ for which~\eqref{e:b2} holds.

Given a triple $\B{r} = (\alpha, \beta, k) \in \B{R}$, we say that a coloured-interval $C$ is \emph{$\B{r}$-good} if its closure $[C]$ contains a (unique) \emph{central} blue interval that is within distance $\alpha|C|$ of each end of $[C]$, and in particular has length at least $(1-2\alpha)|C|$; otherwise, we say that $C$ is \emph{$\B{r}$-bad}. Next, given $L >0$, we say that $C$ is \emph{$(\B{r}, L)$-typical} if
\begin{equation}\label{e:b3}
 L<|C|<\beta L.
\end{equation}

When $\B{r}$ and $L$ are clear from the context, we abbreviate $\B{r}$-good, $\B{r}$-bad and $(\B{r}, L)$-typical to good, bad and typical respectively. We start with the following lemma.

\begin{lemma}\label{l:renorm}
Let $\B{r} = (\alpha, \beta, k)$ be renormalisable and let $L >0$. For any sequence $C_1, C_2, \dots, C_k$ of $(\B{r}, L)$-typical coloured-intervals  with the property that no pair of these coloured-intervals which are at most two apart in the sequence are both $\B{r}$-bad, the coloured-interval $C_1+C_2 + \dots+C_k$ is $\B{r}$-good. Furthermore, if $C_i$ is $\B{r}$-good, then the central blue subinterval of $[C_i]$ is never recoloured in the coalescence process on the coloured-interval $[C_1]+[C_2] + \dots+[C_k]$.
\end{lemma}
\begin{proof}
By replacing each $C_i$ by its closure $[C_i]$ if necessary, we may assume that each $C_i$ is closed. In the following, we write $L_i=|C_i|$ for each $1 \le i \le k$. We make three simple observations.

First, if there are two good coloured-intervals adjacent to one another, say $C_i$ and $C_{i+1}$, then the two central blue intervals of $C_i$ and $C_{i+1}$ are longer than the gap between them in $C_i+C_{i+1}$: indeed, assuming $L_i\le L_{i+1}$, the smallest long blue interval is of length at at least $(1-2\alpha)L_i$ whereas the gap between the two long blue intervals is at most
\[
 \alpha L_i+\alpha L_{i+1}<(\alpha+\alpha\beta)L_i<(2-2\alpha-\beta)L_i
 \le(1-2\alpha)L_i,
\]
using~\eqref{e:b3},~\eqref{e:b1}, and the fact that $\beta\ge1$. Hence, the two central blue intervals of $C_i$ and $C_{i+1}$ eventually coalesce into a single blue interval in $[C_i+C_{i+1}]$. Note that $[C_i+C_{i+1}]$ now contains a blue interval of length at least $(1-\alpha)(L_i+L_{i+1})$.

Next, if a coloured-interval $C_i$ is bad for some $3 \le i \le k-2$, then note that $C_{i-2}$, $C_{i-1}$, $C_{i+1}$ and $C_{i+2}$ must, by assumption, be good. We claim that $[C_{i-2} + \dots +C_{i+2}]$ contains a single blue interval containing both $C_i$ and the central blue intervals of $C_{i-2}$, $C_{i-1}$, $C_{i+1}$ and $C_{i+2}$. Indeed, in this case, we have two long blue intervals of length at least $(1-\alpha)(L_{i-2}+L_{i-1})$ and $(1-\alpha)(L_{i+1}+L_{i+2})$ in $[C_{i-2}+C_{i-1}]$ and $[C_{i+1}+C_{i+2}]$ respectively, formed in each case by coalescing the central blue intervals of these good coloured-intervals. The gap between these two long blue intervals, assuming $L_{i-1}+L_{i-2} \le L_{i+1}+L_{i+2}$, is at most
\begin{align*}
 \alpha L_{i-1}+L_i+\alpha L_{i+1}
 &\le \alpha L_{i-1}+(1+\alpha)\beta\min(L_{i-2},L_{i-1})\\
 &<\alpha L_{i-1}+(2-3\alpha)\min(L_{i-2},L_{i-1})\\
 &\le (1-\alpha)(L_{i-1}+L_{i-2}),
\end{align*}
which demonstrates our claim.

Finally, no central blue interval of a good coloured-interval can ever be recoloured in any complete sequence of recolourings for $C_1 + C_2 + \dots + C_k$. Indeed, suppose not and that the first such central blue interval to be recoloured is that of $C_i$. This, in conjunction with our initial observation about two adjacent good coloured-intervals, would imply that both $C_{i-1}$ and $C_{i+1}$ exist and are bad, contradicting the assumptions of the lemma.

By repeatedly applying these observations, we can show by induction that all the central blue intervals from the good $C_i$, except possibly the ones from $C_1$ or $C_k$ when $C_2$ or $C_{k-1}$ are respectively bad, coalesce into a single blue interval in $[C_1 + C_2 + \dots +C_k]$. The furthest this long blue interval can be from, say, the beginning of $C_1+ C_2 +\dots+C_k$ is at most $L_1+L_2+\alpha L_3$. It is therefore enough to show that
\[
 L_1+L_2+\alpha L_3\le\alpha(L_1+L_2 +\dots+L_k).
\]
Simplifying, we see that it is enough to show that $(1-\alpha)(L_1+L_2)\le\alpha(L_4+L_5+\dots+L_k)$, which follows from~\eqref{e:b2} and the fact that $L_1,L_2\le \beta L_i$ for $i\ge 4$.
\end{proof}

Given distributions $\P_R$ and $\P_B$, for a renormalisable triple $\B{r}\in \B{R}$ and a natural number $n \in \N$, we write $q(n, \B{r})$ for the probability that a coloured-interval $C$ which is the concatenation of $2n$ alternately red and blue intervals whose lengths are independent and have distributions $\P_R$ and $\P_B$ respectively is $\B{r}$-bad. 

\begin{theorem}\label{t:renorm}
Fix a renormalisable triple $\B{r} = (\alpha, \beta, k)$ and a positive integer $n\in \N$. For a pair of probability distributions $\P_R$ and $\P_B$, if there exist sequences $(\eta_t)_{t\ge 0}$ and $(L_t)_{t\ge 0}$ such that the concatenation of $2k^tn$ alternately red and blue intervals, whose lengths are independent and have distributions $\P_R$ and $\P_B$ respectively, is $(\B{r}, L_t)$-typical with probability at least $1-\eta_t$, then blue wins almost surely if $\sum_{t \ge 0} q_t$ converges, where the sequence $(q_t)_{t \ge 0}$ is defined by $q_0=q(n, \B{r})$ and $q_{t+1}=(2k-3)q_t^2+k\eta_t$.
\end{theorem}
\begin{proof}
Recall that $\Delta = \Delta(\P_R, \P_B)$ is constructed as follows: take two  i.i.d.\ sequences $(\CR_i)_{i \in \Z}$ and $(\CB_i)_{i\in\Z}$ of random variables with distributions $\P_R$ and $\P_B$ respectively, and then define $\Delta$ to be the colouring of the real line given by
\[\dots +R_{-1}+B_{-1}+R_0+B_0+R_1+B_1+\dots,\]
where $R_i$ is a red interval with $|R_i|= \CR_i$ and $B_i$ is a blue interval with $|B_i| = \CB_i$ for each $i\in \Z$, and the origin is the boundary-point between $R_0$ and $B_0$.

We now define a complete sequence of recolourings as follows. Let $\Delta_0$ be the colouring obtained from $\Delta$ by replacing blocks of $2n$ consecutive red and blue intervals by their closure; in other words, $\Delta_0$ is given by
\[ \dots+C^{(0)}_{-1}+C^{(0)}_0+C^{(0)}_1+\dots, \]
where the coloured-interval $C^{(0)}_i$ is given by
\[C^{(0)}_i  = \left[R_{in-\lfloor n/2 \rfloor}+B_{in - \lfloor n/2 \rfloor} + \dots + R_{(i+1)n - \lfloor n/2 \rfloor - 1}+B_{(i+1)n - \lfloor n/2 \rfloor- 1}\right]\]
for each $i \in \Z$. Now, having defined $\Delta_t$ to be a colouring of the line into intervals with the representation
\[ \dots+C^{(t)}_{-1}+C^{(t)}_0+C^{(t)}_1+\dots, \]
where $C^{(t)}_{i}$ is a coloured-interval for each $i \in \Z$, we define $\Delta_{t+1}$ to be the colouring obtained from $\Delta_t$ by combining the coloured-intervals of $\Delta_t$ in blocks of size $k$. In other words, $\Delta_{t+1}$ is the colouring of the line into intervals with the  representation
\[ \dots+C^{(t+1)}_{-1}+C^{(t+1)}_0+C^{(t+1)}_1+\dots ,\]
where
\[C^{(t+1)}_i  = \left[C^{(t)}_{ik - \lfloor k/2 \rfloor}+ \dots + C^{(t)}_{(i+1)k - \lfloor k/2 \rfloor - 1}\right].\]

Recall that $k > 6$ for any renormalisable triple, so the union of the coloured-intervals $C_0^{(t)}$ is the entire real line. Consequently, if a monochromatic interval $I$ is recolourable at some stage, then $I$ gets recoloured eventually since $I$ must necessarily be contained in $C_0^{(t)}$ at some time $t \ge 0$; this implies that the sequence of recolourings used to generate the sequence $(\Delta_t)_{t \ge 0}$ is complete.

To finish the proof, we make the following observation.

\begin{proposition}
For all $t\ge 0$ and all $i \in \Z$, the probability that $C^{(t)}_i$ is $\B{r}$-bad is at most $q_t$.
\end{proposition}
\begin{proof}
We proceed by induction on $t$. For $t = 0$ this holds by the definition of $q_0=q(n, \B{r})$. Next, we observe that for all $t \ge 0$, the goodness and the typicality of $C^{(t)}_i$ is independent of $C^{(t)}_j$ for $i\ne j$. Now suppose that $C= [C_i+C_{i+1}+\dots+C_{i+k-1}]$ is a coloured-interval of $\Delta_{t+1}$ obtained by coalescing some $k$ consecutive coloured-intervals $C_i, C_{i+1}, \dots, C_{i+k-1}$ of $\Delta_t$. The probability that two of the coloured-intervals $C_i, C_{i+1},\dots, C_{i+k-1}$ are both bad and at most two apart is at most $(k-1)q_t^2+(k-2)q_t^2=(2k-3)q_t^2$.
Since $\eta_t$ bounds the probability that $C^{(t)}_i$ is not $(\B{r}, L_t)$-typical, the probability that one of $C_i, C_{i+1}, \dots, C_{i+k}$ is not $(\B{r},L_t)$-typical is at most $k\eta_t$. If neither of these events occur, then by Lemma~\ref{l:renorm}, $C=C_1+C_{i+1}+\dots+C_{i+k-1}$ is good. Hence, the probability that $C_i^{(t+1)}$ is $\B{r}$-bad is at most $(2k-3)q_t^2 + k\eta_t = q_{t+1}$, as required.
\end{proof}

We now claim that there is a positive probability that there exists a point which only changes colour a finite number of times and is ultimately blue; the existence of any such point clearly precludes a red-win or a tie, and therefore implies that blue wins almost surely. To see the claim, note that as $\sum_{t \ge 0} q_t$ converges, there exists a $T \ge 0$ such that $\sum_{t \ge T}q_t < 1$. Consequently, the probability that $C_0^{(t)}$ is $\B{r}$-good for each $t \ge T$ is positive since this probability is at least $1 - \sum_{t \ge T}q_t$. The result now follows by noting that if $C_0^{(t)}$ is $\B{r}$-good for each $t \ge T$, then the central blue interval of $C_0^{(T)}$ is never recoloured.
\end{proof}

To apply Theorem~\ref{t:renorm} it is helpful to obtain some bounds on $\eta_t$; the simplest case occurs when the distributions $\P_R$ and $\P_B$ have finite variances.

\begin{theorem}\label{t:renorm2}
Fix a renormalisable triple $\B{r} = (\alpha, \beta, k)$, a pair of probability distributions $\P_R$ and $\P_B$ with finite variances $\sigma_R^2$ and $\sigma_B^2$ respectively, and set 
 \[
c=\frac{(\sigma_R^2+\sigma_B^2)(\beta+1)^2}{(\mu_R+\mu_B)^2(\beta-1)^2}.
\]
For all large enough $n\in \N$ such that the equation
\begin{equation}\label{eq:quad}
x = (2k-3)x^2 + kc/n
\end{equation} 
has real roots, if $q(n,\B{r})$ does not exceed the largest positive root of~\eqref{eq:quad}, then blue wins almost surely.
\end{theorem}
\begin{proof}
For each $t\ge 0$, define $n_t = k^tn$ and take $L_t=2n_t(\mu_R+\mu_B)/(\beta+1)$. If $C$ is a coloured-interval which is the concatenation of $2n_t$ alternately red and blue intervals, then note that $\E[|C|] = n_t (\mu_R + \mu_B)$ and $\V[|C|] = n_t(\sigma_R^2+\sigma_B^2)$. So by Chebyshev's inequality,
\begin{align*}
 \P(|C|\notin(L_t,\beta L_t))
 &=\P\left(||C|-\E[|C|]|> \frac{(\beta-1)\E[|C|]}{\beta+1}\right)\\
 &\le \frac{n_t(\sigma_R^2+\sigma_B^2)(\beta+1)^2}{(n_t(\mu_R+\mu_B)(\beta-1))^2}\\
 &=\frac{c}{n_t}.
\end{align*}

Now, define $\eta_t=c/n_t$ and note that the sequence $(\eta_t)_{t \ge 0}$ decreases exponentially with $t$. Also, observe that since $k > 6$,  if the roots of~\eqref{eq:quad} are real, then they must both lie in the interval $(0,1)$; since $n$ is assumed to be large enough to ensure that the roots of~\eqref{eq:quad} are real, let $Q \in (0,1)$ be the largest root of the equation $x = (2k-3)x^2+kc/n$.

Let $q_0 = q(n,\B{r})$, and consider the sequence $(q_t)_{t \ge 0}$ defined by 
\[q_{t+1}=(2k-3)q_t^2+k\eta_t.\] If $q_0 = q(n,\B{r}) \le Q$, then it is easy to check by induction that the sequence $(q_t)_{t \ge 0}$ is bounded above by $Q$. We claim that the sequence $(q_t)_{t \ge 0}$ in fact decreases exponentially with $t$ provided $q_0 \le Q$; if this holds, then it is easy to see that the result follows from Theorem~\ref{t:renorm}.

To finish the proof, note that if we have $q_t\le Q$ for all $t \ge 0$, then it plainly follows that
\begin{align*}
q_{t} &= (2k-3)q_{t-1}^2+k\eta_{t-1}\\ 
&\le (2k-3)Q q_{t-1} + k\eta_{t-1}\\
&= \zeta q_{t-1} + k\eta_{t-1},
\end{align*}
where $\zeta=(2k-3)Q = 1-(k\eta_0)/Q<1$. Since $\eta_t\to 0$ exponentially as $t \to \infty$, it is now clear that $q_t \to 0$ exponentially as $t \to \infty$; to see this, we expand the above estimate to conclude that 
\[q_t \le k\eta_{t-1}+\zeta k\eta_{t-2}+\dots+\zeta^{t-1} k\eta_0+\zeta^tq_0,\] 
from which it follows that $q_t$ is at most the sum of $\zeta^t q_0$ and a geometric sum which is at most $(\max\{\zeta,1/k\})^{t-1} (tk\eta_0)$.
\end{proof}

\section{Applications of the renormalisation argument}\label{s:renorm}
In this section, we shall apply the renormalisation argument to prove Theorem~\ref{t:mean}, Claim~\ref{t:counter} and Claim~\ref{t:trans}. We will need some Chernoff-type bounds for the tail of the Poisson distribution. We state one such estimate here; see~\citep{probtextbook} for a proof.

\begin{proposition}\label{c:chern}
Let $X$ be a random variable with distribution $\Po(\lambda)$. Then for all $x \le \lambda$,
\[
 \P(X\le \lambda-x)\le \exp\left(-\frac{x^2}{2\lambda}\right),
\]
and for all $x \ge 0$,
\[
 \P(X\ge \lambda+x)\le \exp \left(- \frac{x^2}{2\lambda+2x/3} \right). \eqno\qed
\]
\end{proposition}

We first prove Theorem~\ref{t:mean} which states that for any $K>0$, there exist distributions $\P_R$ and $\P_B$ with $\mu_R > K\mu_B$ where blue wins almost surely in linear coalescence. The basic nature of the example used to demonstrate this is as follows. We choose $\P_R$ so that the vast majority of the contribution to $\mu_R$ comes from exponentially long, exponentially improbable intervals. Meanwhile, $\P_B$ is concentrated on intervals which are slightly longer than almost all of the red intervals. Intuitively, at each stage almost all of the red intervals are short, and are absorbed into the slightly longer adjacent blue intervals; this causes the typical blue interval to now be longer than almost all of the remaining red intervals, and the process repeats. 

\begin{proof}[Proof of Theorem~\ref{t:mean}]
Fix $\B{r}= (\alpha, \beta, k) = (0.23, 1.04, 10)$ and let $n \in \N$ be a sufficiently large natural number. We note that we have taken $\beta$ just above $1$, $\alpha$ just below the bound implied by~\eqref{e:b1}, and $k$ sufficiently large to ensure~\eqref{e:b2}.

We take the blue distribution $\P_B$ to be $U[1,1+\eps]$ where $\eps=1/n^2$. We then set $N = 2K/\eps$ and take the red distribution $\P_R$ to be the distribution of the random variable $1+\sum_{i=1}^N k^iX_i$, where the $X_i$ are independent Poisson random variables such that the mean of $X_i$ is $\eps k^{-i}$. Clearly, $\mu_B=1+\eps/2$ and $\mu_R=1+N\eps = 1 + 2K$, so $\mu_R>K\mu_B$.

Writing $C_t$ for the concatenation of $2k^tn$ alternately red and blue intervals, we note that with probability $1-O(n\eps)=1-O(1/n)$, all the red intervals in $C_0$ are of length $1$, so $C_0$ is $\B{r}$-good (indeed, $[C_0]$ is completely blue apart from a single red interval of length $1$ at one end) with probability $1-O(1/n)$.
Thus, we see that $q(n, \B{r}) = O(1/n)$. Let
\[
 L_t=\frac{2}{\beta+1}\left(2+\frac{\eps}{2}+\min\{t,N\}\eps\right)k^tn
\]
and let $1-\eta_t$ be the probability that $L_t<|C_t|<\beta L_t$. For $t>N$, it follows from Chebyshev's inequality (as in the proof of Theorem~\ref{t:renorm2}) that for any fixed $n$, $\eta_t\to0$ exponentially as $t\to\infty$. We shall show that as $n\to\infty$, $\eta_t \to 0$ \emph{uniformly} in $t$. This would show that if we define $q_0 = q(n, \B{r}) = O(1/n)$ and $q_{t+1} = (2k-3)q_t^2 + k\eta_t$, then as long as $n$ is chosen to be sufficiently large, $\sum_{t\ge 0} q_t$ converges; blue then wins almost surely by Theorem~\ref{t:renorm}.

In the calculations which follow, we shall make use of the fact that since $\beta = 1.04$, we have $2/(\beta+1) < 0.99$ and $2\beta/(\beta+1) > 1.01$.
 
To estimate $\P(|C_t| \le L_t)$, it suffices to estimate the probability that the length of the concatenation of $k^tn$ red intervals is at most $0.99(1+\min\{t,N\}\eps)k^tn$ as the minimum possible length of a blue interval is $1>0.99(1+\eps/2)$. The length of the concatenation of $k^tn$ red intervals is given by the random variable 
\[k^tn+\sum_{i=1}^N k^iY_i,\] 
where the $Y_i$ are independent Poisson random variables such that the mean of $Y_i$ is $\lambda_i=\eps k^{t-i}n$. If $t<1/100\eps=n^2/100$, then this random variable is deterministically larger than $0.99(1+\min\{t,N\}\eps)k^tn$. If $t \ge n^2/100$ on the other hand, then we appeal to Proposition~\ref{c:chern}. Note that 
\[\P(Y_i\le 0.995\lambda_i)\le \exp\left({-10^{-5}\lambda_i}\right),\] 
Writing $t'=\min\{t-2\log_kn,N\}$, the probability that $Y_i\ge 0.995\lambda_i$ for each $1 \le i \le t'$ is at least
\[
1 - \sum_{i=1}^{t'} \exp\left({-10^{-5}\eps k^{t-i}n}\right) \ge 1- N\exp\left({-10^{-5}n}\right),
\]
which is $1-o(1)$ as $n\to\infty$. If this happens, then
\[ k^tn+\sum_{i=1}^N k^iY_i\ge
 k^tn+\sum_{i=1}^{t'}k^iY_i\ge (1+0.995t'\eps)k^tn
\ge 0.99(1+\min\{t,N\}\eps)k^tn,
\]
where the last inequality follows from the fact that $t\ge 0.01n^2$ so that $t-2\log_kn\ge 0.995t$ for all sufficiently large $n$.

To estimate $\P(|C_t| \ge \beta L_t)$, it is enough to show that the total length of the concatenation of $k^tn$ red intervals is very likely to be less than $1.01(1+\min\{t,N\}\eps) k^tn$ as the length of each blue interval is less than $1+\eps<1.01<2\beta(1+\eps/2)/(\beta + 1)$. In other words, it is enough to show that $k^tn+\sum_{i=1}^N k^iY_i$, where the $Y_i$ are independent Poisson random variables such that the mean of $Y_i$ is $\lambda_i=\eps k^{t-i}n$, is very likely to be less than $1.01(1+\min\{t,N\}\eps) k^tn$.

A similar calculation to the one above shows that with high probability, indeed, with probability $1 - N\exp(-n/10^6)$, we have $Y_i<1.005\lambda_i$ for each $i\le t'=t-2\log_k n$.  For $i>t'$, we note that
\[
\sum_{i=t'+1}^{\min\{t,N\}}\E\left[k^iY_i\right]=(\min\{t,N\}-t')\eps k^tn \le N\eps k^t n.
\]
For these values of $i$, we use Markov's inequality to deduce that with probability at least $1-1/\log{n}$, the sum above is
at most $(\log n) N\eps k^t n \le(\log n)^2\eps k^tn$. Thus, with high
probability,
\begin{align*}
 \sum_{i = 1}^{\min\{t,N\}}k^iY_i
 &\le 1.005\min\{t,N\}\eps k^tn +(\log n)^2\eps k^tn\\
 &\le 1.005\min\{t,N\}\eps k^tn+0.01 k^tn.
\end{align*}
Indeed, when $n$ is sufficiently large, $(\log n)^2\eps=n^{-2}(\log n)^2 \le 0.01$. Finally, with probability $1-O(\eps)$, $Y_i=0$ for all $i>t$. 

Putting these together, we see the total length of the red intervals is, with high probability, at most
\[
 k^tn+(1.005\min\{t,N\}\eps k^tn+0.01 k^tn)+0 \le 1.01(1+\min\{t,N\}\eps) k^tn
\]
as required.

In conclusion, if $n$ is sufficiently large then we can take $\eta_t$ to be uniformly small and eventually decreasing exponentially with $t$, and as $q(n,\B{r})$ can be made arbitrarily small by taking $n$ to be sufficiently large, we see that blue wins by Theorem~\ref{t:renorm} for all suitably large $n$.
\end{proof}

The advantage of the renormalisation argument is that it allows us to deduce the outcome of linear coalescence from the typical outcome of the coalescence process on a large finite coloured-interval. The main difficulty in applying Theorem~\ref{t:renorm} or~\ref{t:renorm2} with some fixed renormalisable triple $\B{r} \in \B{R}$ is that we need to understand the coalescence process on a large finite coloured-interval reasonably well; in particular, we need good estimates for the probability $q(n,\B{r})$ that the concatenation of $2n$ alternately red and blue intervals is $\B{r}$-bad. We need $q(n,\B{r})$ to be quite small for the renormalisation argument to be useful which means that in practice, we need to take $n$ to be quite large; but this in turn often makes the task of proving a useful bound on $q(n,\B{r})$ impractical.

However, we can use a Monte Carlo approach to obtain high confidence results. Namely, we simulate linear coalescence on the concatenation of $2n$ alternately red and blue intervals many times and count the number of times the resulting interval is good. We can deduce that if $q(n,\B{r})$ was too big, then the probability of obtaining these simulated results is incredibly small, assuming of course that the random number generator used in the simulation adequately resembles real random numbers and that no errors have occurred during the programming or execution of the computer program. 

We shall use this Monte Carlo approach to demonstrate Claims~\ref{t:counter} and~\ref{t:trans}.

\begin{proof}[Proof of Claim~\ref{t:counter}]
Fix $\B{r}= (\alpha, \beta, k) = (0.23, 1.04, 10)$ and let $n_0=2 \times 10^6$. Next, fix $c_1=0.08$ and $c_2=0.01$. 

We first construct a pair of distributions $\P_R$ and $\P_B$ such that $\P_R \succ \P_B$ for which blue wins almost surely with very high confidence. We then sketch how we may, as in the proof of Theorem~\ref{t:main}, `blow up' the mean of the red distribution without changing the outcome of the coalescence process. 

The main obstacle then is to allow stochastic domination. The intuition behind our construction is that the blue intervals being short has two opposing effects. While short blue intervals are less likely to absorb adjacent red intervals, they also contribute less to any red intervals which absorb them. Hence, any blue interval which is likely to be absorbed may as well be short. Of course, this intuition is insufficient, as the marginal shortening of the red intervals may, in turn, prevent blue intervals from growing quickly.

We take the red distribution $\P_R$ to be uniform on $[1,1+c_2]$ and the blue distribution $\P_B$ to be uniform on $[0,c_1c_2]\cup[1+c_1c_2,1+c_2]$. Note that we can obtain the blue distribution by sampling from the red distribution and subtracting $1$ from the sampled length if it is less than $1+c_1c_2$; it is then clear that $\P_R$ stochastically dominates $\P_B$ strictly.

We would like to apply Theorem~\ref{t:renorm2}. Elementary calculations give
\begin{align*}
 \mu_B&=1+\frac{c_2}{2}-c_1,\\
 \mu_R&=1+\frac{c_2}{2},\\
 \sigma^2_B&=c_1(1-c_1)(1+c_2)+\frac{c_2^2}{12}, \text{ and}\\
 \sigma^2_R&=\frac{c_2^2}{12}
 \end{align*}
To apply Theorem~\ref{t:renorm2}, we need $q(n_0, \B{r})$ to be less than the largest positive root of the equation $(2k-3)x^2+kc/n_0 = x$, where $k=10$, $n_0=2\times10^6$ and
\[
 c=51^2 \times \frac{c_1(1-c_1)(1+c_2)+c_2^2/6}{(2+c_2-c_1)^2}.
\]
A simple calculation shows that it is sufficient to show that $q(n_0, \B{r}) < 0.058$. To estimate $q(2\times10^6, \B{r})$, the coalescence process on the concatenation of $4\times 10^6$ alternately red and blue intervals was simulated 1000 times. The coloured-intervals obtained were $\B{r}$-good in 987 trials out of the total of 1000 trials performed. If $q \ge 0.058$, then the probability of obtaining at least 987 good coloured-intervals in 1000 trials is less than $10^{-12}$. Hence, with very high confidence, blue wins almost surely.

We now sketch how we can mimic the proof of Theorem~\ref{t:mean} to show that we may alter the red distribution to make its mean arbitrarily large without changing the fact that blue wins. 

For a given $K > 0$ and an integer $n \ge 0$, let $\P_{R,n}$ denote the distribution of the random variable $\sum_{i=0}^{N} k^iX_i$, where $N = 2Kn$ and the $X_i$ are independent random variables such that the distribution of $X_0$ is $\P_R$ and that of $X_i$ is $\Po(1/k^in)$ for each $1 \le i \le N$. Note that $\P_{R,0} = \P_R$. Note also that for each $n \ge 1$, the ratio of the means of $\P_{R,n}$ and $\P_B$ is at least $K$ and since $\P_{R,n} \succ \P_R$, we also have $\P_{R,n} \succ \P_B$. We claim that in linear coalescence where the red and blue distributions are $\P_{R,n}$ and $\P_B$, blue wins almost surely if $n$ is sufficiently large.

For each $t \ge 0$ and $n \ge 0$, let
\[
L_{t,n} = \frac{2}{\beta + 1}(\mu_R + \mu_B +\min\{t,N\}/n)k^tn_0
\]
and define $\eta_{t,n}$ to be the probability that concatenation of $2k^tn_0$ alternately red and blue intervals whose lengths have distributions $\P_{R,n}$ and $\P_B$ respectively is not $(\B{r}, L_{t,n})$-typical. Finally, we define the sequence $(q_{t,n})_{t \ge 0}$ by setting
\[
q_{t+1,n}=(2k-3)q_{t,n}^2+k\eta_{t,n},
\]
for each $t\ge 0$, where $q_{0,n}$ is the probability that the concatenation of $2n_0$ alternately red and blue intervals whose lengths have distributions $\P_{R,n}$ and $\P_B$ is $\B{r}$-bad.

We know that $q_{0,0} < 0.058$ with very high confidence, and in this case Theorem~\ref{t:renorm2} tells us precisely that $\sum_{t \ge 0} q_{t,0}$ converges. As in the proof of Theorem~\ref{t:renorm2}, we can show that $\eta_{t,n}$ decreases exponentially with $t$ for any fixed $n \ge 0$. Furthermore, we can show mimicking the proof of Theorem~\ref{t:mean} that as $n \to \infty$, $\eta_{t,n} \to \eta_{t, 0}$ uniformly in $t$. Finally, it is clear that $q_{0,n} = q_{0,0} + O(1/n)$. It follows that for any $K > 0$, the sum $\sum_{t \ge 0} q_{t,n}$ converges if $n$ is chosen to be suitably large. By Theorem~\ref{t:renorm2}, blue wins almost surely.
\end{proof}

We now turn to the proof of Claim~\ref{t:trans}.

\begin{proof}[Proof of Claim~\ref{t:trans}]
Let $\P_R$ be the distribution which is deterministically $1$, $\P_G$ be the (exponential) $\Exp(1.22)$ distribution and $\P_B$ the (uniform) $U[0,2.19]$ distribution. As before, fix $\B{r}= (\alpha, \beta, k) = (0.23, 1.04, 10)$ and let $n_0=2 \times 10^6$.

\begin{figure}
	\begin{subfigure}{\textwidth}
		\centering
		\includegraphics[width= 0.95\textwidth, height = 4.0cm]{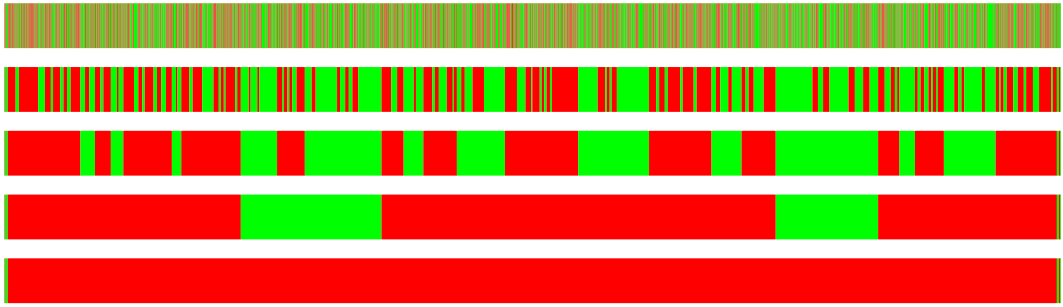}
		\caption{$\P_R \rhd \P_G$.}
		\vspace{0.5cm}
	\end{subfigure}
	\begin{subfigure}{\textwidth}
		\centering
		\includegraphics[width= 0.95\textwidth, height = 4.0cm]{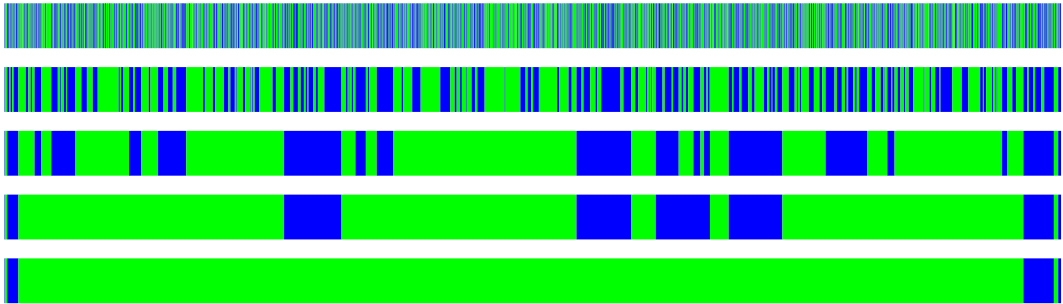}
		\caption{$\P_G \rhd \P_B$.}
		\vspace{0.5cm}
	\end{subfigure}
	\begin{subfigure}{\textwidth}
		\centering
		\includegraphics[width= 0.95\textwidth, height = 4.0cm]{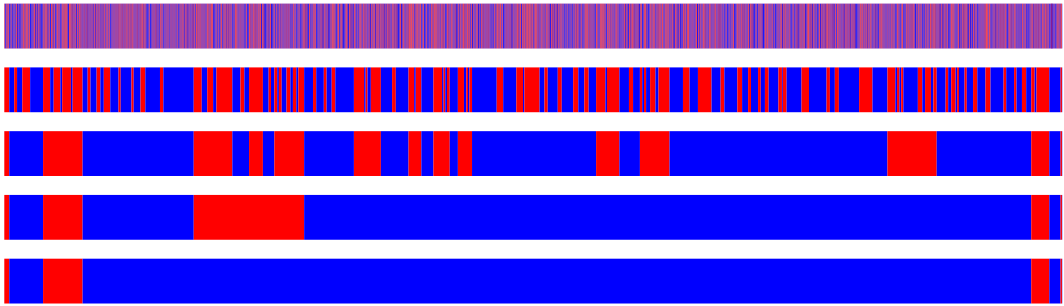}
		\caption{$\P_B \rhd \P_R$.}
	\end{subfigure}
	\caption{Coalescence on the line is intransitive.}
	\label{intrans}
\end{figure}

Let $q_{BR}(n_0, \B{r})$ denote the probability that the concatenation of $2n_0$ alternately red and blue intervals is $\B{r}$-bad. Define $q_{RG}(n_0, \B{r})$ and $q_{GB}(n_0, \B{r})$ analogously (where in each case, we ask for a long central interval of the appropriate colour in the closure).

To show that $\P_R \rhd \P_G$, $\P_G \rhd \P_B$ and $\P_B \rhd \P_R$, by Theorem~\ref{t:renorm2}, it is sufficient to verify, with $n = 2\times 10^6$, that $q_{RG}(n, \B{r}) < 0.0625$, $q_{GB}(n, \B{r}) < 0.063$  and $q_{BR}(n, \B{r}) < 0.0599$. These inequalities were verified using Monte Carlo methods with very high confidence.
\end{proof}

\section{The \texorpdfstring{$\ell$}{ell}-bounding argument}\label{s:strat}
The renormalisation argument has two main disadvantages. First, it requires us to estimate certain large finite-dimensional numerical integrals fairly precisely. Second, when we wish to obtain a high confidence result using the renormalisation argument, we can only do so for a fixed pair of distributions; in general, it is not possible to use these techniques to compare different \emph{families} of distributions since, as we have seen, linear coalescence is far from monotone. With a view of getting around these difficulties, in this section, we shall introduce a method for tracking the coalescence process by maintaining a collection of approximations to the lengths of the intervals.

Suppose that we wish to show that blue wins, but are unable to follow the process precisely. It is possible to approximate the process in a way that is `pessimistic for blue' so that if blue still wins in this setting, then we can deduce that blue wins in the original process. The first observation is that if we occasionally make mistakes and recolour blue intervals not surrounded by longer red intervals, then this is always pessimistic for blue.

\begin{lemma}\label{l:bias}
Suppose that we coalesce intervals using a rule that only recolours red intervals when surrounded by longer blue intervals and always recolours blue intervals when surrounded by longer red intervals (but may sometimes recolour blue intervals when this does not hold). If blue wins in this new process, then it wins in the original coalescence process.
\end{lemma}
\begin{proof}
Imagine two copies of the line $\R_1$ and $\R_2$ with the same sequence of monochromatic intervals on both. We run the new process on $\R_1$, and at each step, if we recolour an interval $I$ using it's neighbours $I_-$ and $I_+$ on $\R_1$, then we replace the coloured-interval corresponding to $I_- + I + I_+$ by its closure on $\R_2$.

We show by induction that at each step, a blue interval in $\R_1$ corresponds to a blue interval in $\R_2$, and a red interval in $\R_1$ corresponds to a red-ended coloured-interval, i.e., a coloured-interval that starts and ends with a red subinterval, in $\R_2$.

Suppose that this holds at time $t-1$, and suppose that we recolour a blue interval $B$ red in $\R_1$ at time $t$. Then in $\R_2$, we obtain (even without recolouring) a red-ended interval $C_- + B\ + C_+$, where $C_\pm$ are the red-ended intervals corresponding to the red neighbours of $B$ in $\R_1$. Also, any valid recolouring of the internal subintervals of $C_- + B\ + C_+$ results in a red-ended interval. 

Now suppose that we recolour a red interval $R$ blue in $\R_1$. Then in $\R_1$, this gives a blue interval which is the closure of $B_- + R + B_+$ where $B_\pm$ are the blue neighbours of $R$ and $|B_\pm|>|R|$. In $\R_2$, this corresponds to taking the closure of a sequence of intervals of the form $B_- + R_1 + B_1 + \dots + R_{2k+1} + B_+$ with $B_\pm$ being the longest two subintervals. As the shortest subinterval is internal, it can be recoloured reducing the number of intervals by two. Repeating this process and noting that the outermost intervals are both always longer than the sum of all the internal intervals, we see that we can perform a sequence of valid recolourings reducing this sequence to a single blue interval as required. 

Therefore, if blue wins in the new process on $\R_1$, then it does so on $\R_2$ as well; the result follows from this fact combined with Lemma~\ref{l:unique}.
\end{proof}

Although Lemma~\ref{l:bias} is useful, in practice, we need to approximate lengths, rather than  approximating the decisions on whether or not to recolour (while maintaining the exact lengths). A natural idea is to approximate the lengths by always underestimating the lengths of blue intervals and overestimating the lengths of red intervals. However, we naturally run into a problem when we recolour as we are always underestimating some of the constituent lengths and overestimating others. Thus, we cannot tell if the recoloured interval length is an overestimate or an underestimate.

One solution is to track, for each interval $I$, a range of possible lengths with $\ell_-(I)\le |I|\le \ell_+(I)$. We base the decision to recolour an interval $I$ on the underestimate $\ell_-(I)$ if $I$ is blue, and on the overestimate $\ell_+(I)$ is $I$ is red. If a blue interval $B$ is surrounded by red intervals $R_\pm$ that are possibly longer, so that $\ell_+(R_\pm)\ge \ell_-(B)$, then we recolour $B$. Similarly, if a red interval $R$ is surrounded by blue intervals $B_\pm$ that are definitely longer, so that $\ell_-(B_\pm)>\ell_+(R)$, then we recolour $R$. The minimum and maximum lengths of the resulting recoloured interval is then obtained by adding the minimum or maximum lengths of all the constituent intervals. In practice, the errors in the lengths of the intervals grows quickly, so this procedure can generally only be applied for a few steps before other methods are required.

To understand the evolution analytically, the following observation is crucial. Suppose that the minimum length of an interval is $L$ and occurs with positive probability $p$ in $\P_B$, say. Then recolouring all blue intervals of length $L$ results in a new colouring of $\R$ where the lengths of the red and blue intervals are still independent. The blue distribution is replaced by the same distribution conditioned on the length being greater than $L$, while the red distribution is replaced by the distribution of the random variable $\sum_{k=1}^{Y} X_i+ (Y-1)L$,
where the $X_i$ are i.i.d.\ random variables with distribution $\P_R$ and $Y$ is a (geometric) $\Geom(p)$ random variable. Indeed, it is easily seen that we just recolour any sequence of red intervals where all the intervening blue intervals are of length $L$, and as we travel along the line, these groups occur independently and include $\Geom(p)$ red intervals.

If the minimum length does not occur as an atom, one can discretise the distributions and use the length bounding approach described above. By making the discretisation finer and finer it is possible to track the result of recolouring all intervals of less than some length via a differential equation in terms of the distributions $\P_R$ and $\P_B$. 

If we write $f_R(x,t)$ and $f_B(x,t)$ for the probability density functions of the two distributions after all intervals of length at most $t$ have been eliminated, one obtains an evolution of the following form.
\begin{align}
 \frac{\partial}{\partial t} f_R(x,t)&= (f_R(t,t)-f_B(t,t))f_R(x,t)\notag\\
 &+\one_{\{x>3t\}}f_B(t,t)\int_t^{x-2t}f_R(z,t)f_R(x-t-z)\,dz,\notag\\
  \frac{\partial}{\partial t} f_B(x,t)&= (f_B(t,t)-f_R(t,t))f_B(x,t)\notag\\
 &+\one_{\{x>3t\}}f_R(t,t)\int_t^{x-2t}f_B(z,t)f_B(x-t-z)\,dz. \label{e:diff}
\end{align}
Indeed, in time $dt$, there is a density $f_B(t,t)dt$ of blue intervals that are recoloured red and the conditional distribution of the rest becomes $f_B(x,t)/(1-f_B(t,t)dt)$. On the other hand, the blue
intervals are grouped into blocks of $\Geom(f_R(t,t)dt)$ intervals separated by red blocks of length about $t$. To order $dt$, this is equivalent to replacing a fraction $f_R(t,t)dt$ blue intervals by random intervals of length $X_1+t+X_2$ where the distributions of $X_1$ and $X_2$ is $\P_B$.

Tracking the evolution of~\eqref{e:diff} seems impractical, particularly if one is interested in rigorous results. In particular, it is not clear what the `endgame' is as even if~\eqref{e:diff} could be tracked reasonably accurately up to some large time; the renormalisation argument fails to be of help as the distributions of the lengths typically never become concentrated enough. Unless we can run~\eqref{e:diff} up to infinity, it is not clear which colour wins. Thus, it is unclear whether anything can be deduced from just an approximate version of the probability distributions, and solving~\eqref{e:diff} exactly seems out of reach for any nontrivial pair of starting distributions.

To make progress, we therefore need some way of bounding the process without accurate information on the lengths of the intervals. As remarked above, it is not enough just to na{\"i}vely bound the lengths of the blue intervals from below and the lengths of the red intervals from above. However, it is possible make some headway if we use a more cautious method of
combining intervals.

To do this we first prove a bound on the `red-content' of a recoloured sequence. We call a coloured-interval $C$ a \emph{red-ended interval} if it starts and ends with a red subinterval; in other words, $C=R_1+B_1+R_2+\dots+B_{k-1}+R_k$ for some red intervals $R_i$, blue intervals $B_j$, and $k\ge1$. We note that if $C$ is red-ended, then so is its closure $[C]$. For a red-ended interval $C$, define $r(C)$ to be the total length of the red subintervals in the closure of $C$. 

\begin{lemma}\label{l:2r}
If $C_-$ and $C_+$ are red-ended intervals and $B$ is a blue interval, then we have $r(C_-+B+C_+)\le 2r(C_-)+2r(C_+)$.
\end{lemma}
\begin{proof}
We may assume without loss of generality that $C_-$ and $C_+$ are already closed. Hence, let $C_-=R_1+B_1+\dots+R_i$ with the lengths of the subintervals forming a unimodal sequence, and let $C_+ = R_{i+1}+B_{i+1}+\dots+R_{i+j}$ similarly. 

We prove the result by induction on $i+j$. If  $i+j\le2$, then $i=j=1$, and $C_-$ and $C_+$ are red intervals. Now, $C_-+B+C_+$ is closed if and only if $|B|>|C_-|$ or $|B|>|C_+|$; otherwise, $B$ can be recoloured red. Thus, 
\begin{align*}
r(C_-+B+C_+)&\le |C_-|+|C_+|+\min\{|C_-|,|C_+|\}\\
&\le 2|C_-|+2|C_+|=2r(C_-)+2r(C_+).
\end{align*}
Hence, we may assume $i+j>2$.

\textbf{Case 1: $i\ge 2$ and $|R_1|<|B_1|$.}
Let $C'_-=R_2+B_2+\dots+R_i$ and note that since $C'_-$ is closed, $r(C'_-) = r(C_-) - |R_1|$. We may assume inductively that $r(C'_-+B+C_+)\le 2r(C'_-)+2r(C_+)$. In $R_1+B_1+[C'_-+B+C_+]$, the only possible initial recolouring is of the first red interval of $[C'_-+B+C_+]$ which results in a blue interval containing $B_1$. However, $B_1$ can never be recoloured as $|B_1|>|R_1|$. Thus, although it is possible for red intervals to be recoloured, no blue intervals will ever be recoloured in the coalescence process on $R_1+B_1+[C'_-+B+C_+]$.
Thus, 
\begin{align*} 
r(C_-+B+C_+)&\le |R_1|+r(C'_-+B+C_+)\\
&\le 2(|R_1|+r(C'_-))+2r(C_+)\\
&=2r(C_-)+2r(C_+).
\end{align*}
A similar proof also holds for $j\ge 2$ and $|B_{i+j-1}|>|R_{i+j}|$. Hence, we may assume the lengths of the subintervals in $C_-$ are decreasing and the lengths of the subintervals in $C_+$ are increasing.

\textbf{Case 2: $i\ge 2$ and $|B|>|R_i|$.}
Let $C'_-=R_1+B_1+\dots+R_{i-1}$ and $B'=B_{i-1}+R_i+B$. Note that $[B']$ is blue and $r(C'_-) = r(C_-) - |R_i|$ since $C'_-$ is closed. Therefore, it follows that
\[ r(C_-+B+C_+)=r(C'_-+[B']+C_+)\le 2r(C'_-)+2r(C_+)\le 2r(C_-)+2r(C_+).\]
A similar proof also works if $|B| > |R_{i+1}|$, so we may assume that $|B|<|R_i|$ and $|B| < |R_{i+1}|$.

\textbf{Case 3: $|R_1|>\dots>|R_i|>|B|<|R_{i+1}|<\dots<|R_{i+j}|$.}
In this case, it is easy to check that the total length of the red intervals in $C_-+B+C_+$ is at least half the total length of $C_-+B+C_+$. Thus, 
\[r(C_-+B+C_+)\le |C_-+B+C_+|\le 2r(C_-)+2r(C_+). \qedhere\]
\end{proof}

The following is an easy corollary of Lemma~\ref{l:2r}.

\begin{corollary}\label{c:multi}
If $C_1, C_2, \dots,C_k$ are red-ended intervals and $B_1, B_2, \dots,B_{k-1}$ are blue intervals, then 
\[ r(C_1+B_1+ C_2+\dots+ B_{k-1}+C_k)\le 2^{\lceil\log_2k\rceil}\sum_{i=1}^k r(C_i).\]
\end{corollary}
\begin{proof}
We use induction on $k$. The result for $k\le2$ follows from Lemma~\ref{l:2r}. Assume now that $k>2$ and write $t=\lceil\log_2k\rceil$ so that $k\le 2^t$. Take $k' = 2^{t-1}$, thereby ensuring that $1\le k',k-k'\le 2^{t-1}$. By induction, we know that 
\[ r(C_1+B_1+\dots+C_{k'})\le 2^{t-1}\sum_{i=1}^{k'}r(C_i)\] 
and that 
\[r(C_{k'+1}+ B_{k'+1} + \dots+C_k)\le 2^{t-1}\sum_{i=k'+1}^k r(C_i).\] 
Thus, it follows that
\begin{align*}
 r(C_1+ B_1 + \dots+C_k)&=r([C_1+\dots+C_{k'}]+B_{k'}+[C_{k'+1}+\dots+C_k])\\
 &\le 2r([C_1+\dots+C_{k'}])+2r([C_{k'+1}+\dots+C_k])\\
 &\le 2^t\sum_{i=1}^k r(C_i).\qedhere
\end{align*}
\end{proof}

Note that we cannot replace the factor of $2^{\lceil\log_2k\rceil} = \Theta(k)$ in Corollary~\ref{c:multi} by an absolute constant. Indeed, there is a simple construction similar to that of the Cantor set which demonstrates this. Start with a unit red interval, and replace (slightly less than) its middle third by a blue interval and inductively repeat this construction in the left and right red subintervals. Plainly, the closure of this sequence is entirely red. Hence, for each $\eps > 0$ and each $k\in \N$ which is a power of $2$, we have demonstrated the existence of red intervals  $R_1, R_2, \dots,R_k$ and blue intervals $B_1, B_2, \dots,B_{k-1}$ for which 
\[r(R_1+B_1+\dots+R_k)\ge (1-\eps)k^{\log_23}\sum_{i=1}^k r(R_i).\]

Finally, we need the following simple observation that complements Lemma~\ref{l:2r}.

\begin{proposition}\label{p:bb}
If $B_-$ and $B_+$ are blue intervals and $C$ is a red-ended interval such that $|B_\pm| > r(C)$, then the closure of $B_- + C + B_+$ is a blue interval.
\end{proposition}
\begin{proof}
We may assume, by replacing $C$ by its closure if necessary, that $C$ is closed. Now, taking the closure of $B_- + C+ B_+$ corresponds to taking the closure of a sequence of intervals of the form $B_- + R_1 + B_1 + \dots + R_{2k+1} + B_+$, where we know that $|B_\pm| > r(C) \ge |R_1|, |R_{2k+1}|$ and that either $|R_1|<|B_1|$ or $|R_{2k+1}|<|B_{2k}|$. It follows by induction that the red intervals in $C$ can be recoloured one at a time, from the outside in, noting that length of the two outermost blue intervals always exceeds the red-content of the red-ended interval between them. It is now clear that $[B_- + C + B_+]$ is a single blue interval, as required.
\end{proof}

We are now in a position to describe our strategy for following the coalescence process analytically. Suppose that we are given a colouring $\Delta$ of the line into intervals, along with a representation of the colouring as  
\[ \dots+C_{-1}+B_{-1}+C_0+B_0+C_1+B_1+\dots, \]
where $C_i$ is a red-ended interval and $B_i$ is a blue interval for each $i \in \Z$. Suppose also that we are given two sequences of bounds $(\ell(C_i))_{i \in \Z}$ and $(\ell(B_i))_{i \in \Z}$ such that $\ell(C_i)\ge r(C_i)$ and $\ell(B_i)\le |B_i|$ for each $i \in \Z$. 

We fix a constant $\ell_0>0$ and update the given colouring and the associated bounds as follows. We first coalesce all the blue intervals $B$ with $\ell(B) \le \ell_0$ with the surrounding red-ended intervals so that we no longer have blue intervals $B$ with $\ell(B) \le \ell_0$. More precisely, if we encounter a sequence 
\[ B_{i-1}+C_i+\dots+C_j+B_{j}\] 
in the colouring with $\ell(B_{i-1}), \ell(B_{j}) > \ell_0$ and $\ell(B_k) \le \ell_0$ for every $i \le k < j$, then we replace the sequence $C_i+\dots+C_j$ in the colouring by its red-ended closure $C = [C_i+\dots+C_j]$. In the updated colouring, we set 
\[ \ell(C) = 2^{\lceil\log_2 (j-i+1)\rceil}\sum_{k=i}^j\ell(C_k),\]
noting that Corollary~\ref{c:multi} ensures that $\ell(C) \ge r(C)$. In the resulting colouring, we no longer have blue intervals $B$ with $\ell(B) \le \ell_0$. We update this resulting colouring again as follows: all red-ended intervals $C$ with $\ell(C) \le \ell_0$ are coalesced with the surrounding blue intervals. More precisely, if we have a sequence, 
\[ C_i+B_{i}+\dots+B_j+C_{j+1} \] 
in the colouring with $\ell(C_i), \ell(C_{j+1}) > \ell_0$ and $\ell(C_k) \le \ell_0$ for each $i < k \le j$, then we replace the sequence $B_i+\dots+B_j$ by its closure $B$. It is easy to see from Proposition~\ref{p:bb} that $B$ is a monochromatic blue interval since $|B_k| \ge \ell(B_k) > \ell_0$ for each $i \le k \le j$ and $r(C_k) \le \ell(C_k) \le \ell_0$ for $i+1 \le k \le j$; in the updated colouring, we then set
\[ \ell(B) = \sum_{k=i}^j\ell(B_k),\]
noting that this ensures that $|B| \ge \sum_{k=i}^j|B_k| \ge \sum_{k=i}^j\ell(B_k) = \ell(B)$.
This completes our update of the original colouring and the associated bounds. Note that we have ensured that  $\ell(C)\ge r(C)$ and $\ell(B)\le |B|$ for each red-ended interval $C$ and each blue interval $B$ in the updated colouring; furthermore, note that we also have $\ell(C), \ell(B) > \ell_0$ for each red-ended interval $C$ and each blue interval $B$ in the updated colouring.

We shall refer to this approach to updating a given colouring $\Delta$ and the associated bounds $\ell(.)$ as the \emph{$\ell$-bounding argument (with threshold $\ell_0$)}. Let us note that if the original sequences of bounds $(\ell(C_i))_{i \in \Z}$ and $(\ell(B_i))_{i \in\Z}$ are i.i.d.\ sequences of random variables, then the two new sequences of bounds obtained after updating the colouring using the $\ell$-bounding argument also have the same property, though of course, the two new sequences might now be distributed according to a different pair of distributions; note in particular that the law of the updated sequences of bounds continues to be shift-invariant in this case. Let us also observe that the $\ell$-bounding argument is stable under stochastic domination, so we are free to replace the bounds $\ell(C_i)$ by anything stochastically larger, and the bounds $\ell(B_i)$ by anything stochastically smaller when we use the argument.

We also note that the coalescence process is unaffected by scaling the length of every interval by any positive constant. In particular, if both $\P_R$ and $\P_B$ have support bounded away from 0, then we may assume without loss of generality  that they both have support contained in $[1, \infty)$. Furthermore, if we use the $\ell$-bounding argument to remove all intervals of length less than, say $1+\eps$, we can then divide the length of each interval by $1+\eps$ without altering the future evolution of the process. We will use this idea to make later calculations more tractable.

\section{Proof of the main result}\label{s:main}

In this section, we use the $\ell$-bounding argument to prove our main result, Theorem~\ref{t:main}. Recall the distribution $\Scr{G}(a)$ with density function $2(a+1)^2/(a+x)^3$ for all $x\ge 1$. Our proof of Theorem~\ref{t:main} hinges on the following lemma.

\begin{lemma}\label{l:reddom}
There exist $0<\eps_0<1$ and $\Lambda<14$ such that for all $0\le a\le 1$ and $0<\eps\le \eps_0$, the random variable
\[ 2^{\lceil\log_2Y\rceil}\sum_{i=1}^YX_i\] 
is stochastically dominated by a $\Scr{G}(a+\Lambda\eps)$ random variable, where the $X_i$ and $Y$ are independent random variables such that the distribution of $Y$ is $\Geom(\eps)$ and that of $X_i$ is $\Scr{G}(a)$ for each $1 \le i \le Y$.
\end{lemma}
\begin{proof}
Let $Z=2^{\lceil\log_2Y\rceil}\sum_{i=1}^YX_i$ and let $W$ be a random variable with distribution $\Scr{G}(a+\Lambda\eps)$. We need to show that for all $x\ge1$,
$\P(Z\ge x)\le \P(W\ge x)$.

We start by estimating $\P(W\ge x)$. Observe that
\begin{align*}
 \left(\frac{\P(W\ge x)}{\P(X_1\ge x)}\right)^{1/2}
 &=\frac{a+\Lambda\eps+1}{a+\Lambda\eps+x} \times \frac{a+x}{a+1}\\
 &=1+\frac{\Lambda\eps(x-1)}{(a+\Lambda\eps+x)(a+1)}.
\end{align*}
Hence, it follows that
\begin{equation}\label{e:w}
 \P(W\ge x)=\P(X_1\ge x)
 \left(1+\frac{\Lambda\eps(x-1)}{(a+\Lambda\eps+x)(a+1)}\right)^2.
\end{equation}

Now assume $x\in[1,4]$. As $X_i\ge1$, $Z\ge4$ whenever $Y>1$. Thus, 
\[ \P(Z\ge x)=(1-\eps)\P(X_1\ge x)+\eps,\] 
so we need to show that
\[
 1-\eps+\eps\frac{(a+x)^2}{(a+1)^2}\le
 \left(1+\frac{\Lambda\eps(x-1)}{(a+\Lambda\eps+x)(a+1)}\right)^2
\]
or equivalently
\[
 1+\eps\frac{(x-1)(2a+1+x)}{(a+1)^2}\le 1+2\frac{\Lambda\eps(x-1)}{(a+\Lambda\eps+x)(a+1)}
 +\frac{\Lambda^2\eps^2(x-1)^2}{(a+\Lambda\eps+x)^2(a+1)^2}.
\]
Simplifying, it is enough to show that
\[
 \frac{2a+1+x}{(a+1)^2}\le \frac{2\Lambda}{(a+\Lambda\eps+x)(a+1)}
 +\frac{\Lambda^2\eps(x-1)}{(a+\Lambda\eps+x)^2(a+1)^2}.
\]
As the region $[1,4]\times[0,1]$ of possible values of $(x,a)$ is compact, and since any bound we obtain on $\eps$ will be continuous, it is enough to prove pointwise that for any such $(x,a)$, there is a sufficiently small $\eps$ that satisfies this inequality. As a result it is enough to show that
\[
 \frac{2a+1+x}{(a+1)^2}<\frac{2\Lambda}{(a+x)(a+1)}.
\]
This reduces to the inequality $(2a+1+x)(a+x)<2\Lambda(a+1)$ which holds for all $x\le 4$ if $(2a+5)(a+4)<2\Lambda(a+1)$. This in turn holds for all $a\in[0,1]$ when $\Lambda>10$.

Assume now that $x>4$. We give a proof for $\Lambda=52$ and indicate at the end how to reduce this bound to $\Lambda = 13.06207 <14$.

If $Y\ge2$, then $2^{\lceil\log_2Y\rceil}\le 2(Y-1)$. Hence, if $Z\ge x$ and $Y\ge2$, then at least one of $X_1, X_2, \dots, X_Y$ is greater than $x/(2Y(Y-1))$. Thus
\begin{align}
 \P(Z\ge x)&\le (1-\eps)\P(X_1\ge x)
 +\sum_{k=2}^\infty(1-\eps)\eps^{k-1}k\P(X_i\ge x/(2k(k-1)))\notag\\
 &\le (1-\eps)\P(X_1\ge x)
 +\sum_{k=2}^\infty(1-\eps)\eps^{k-1}k\frac{(a+1)^2}{(a+x/(2k(k-1)))^2}\notag\\
 &\le (1-\eps)\P(X_1\ge x)
 +\sum_{k=2}^\infty(1-\eps)\eps^{k-1}4k^3(k-1)^2\frac{(a+1)^2}{(a+x)^2}\notag\\
 &= (1+31\eps+O(\eps^2))\P(X_1\ge x).\label{e:z}
\end{align}
Now $a\le 1$ and $x\ge 4$, so $(x-1)/((a+x)(a+1))\ge 3/10$. Thus, by~\eqref{e:w}, we have
\begin{equation}\label{e:w2}
 \P(W\ge x)= (1+3\Lambda\eps/5+\Omega(\Lambda^2\eps^2))\P(X_1\ge x).
\end{equation}
Moreover, the constants implicit in the asymptotic notation in both~\eqref{e:z} and~\eqref{e:w2} are absolute constants independent of $a$ and $x$ for any $\eps<1$. Thus, for $\Lambda=52>31(5/3)$, we have $\P(Z\ge x)\le \P(W\ge x)$ for all $x\ge 4$ when $\eps$ is sufficiently small. 

To improve this bound, it is clearly enough to ensure the factor in front of the $\eps$ in~\eqref{e:w2} is larger that the factor in front of the $\eps$ in~\eqref{e:z}. Thus, we need that for some fixed $\delta>0$,
\begin{equation}\label{e:1}
 \frac{\P(X_1+X_2\ge x/2)}{\P(X_1\ge x)}-1+\delta \le\frac{2(x-1)\Lambda}{(a+1)(a+x)}
\end{equation}
for all $a\in[0,1]$ and $x\ge4$. We first evaluate $\P(X_1+X_2\ge x/2)$ in closed form for $x\ge 4$. This allows us to avoid numerical integration and ensures fast and accurate calculation of both sides of~\eqref{e:1}. It is not difficult to show that
\begin{multline*}
 \P(X_1+X_2\ge x/2)
 =\frac{8(a+1)^2((4a+x)^2+6(a+1)(x-4))}{(4a+x)^3(2a+x-2)}
 \\+\frac{192(a+1)^4}{(4a+x)^4}\log\left(\frac{2a+x-2}{2a+2}\right).
\end{multline*}

Next, note that $\P(X_i\ge x)$ is not only monotone in $x$, it is also increasing in $a$; also, the same is true of
$\P(X_1+X_2\ge x/2)$. Thus, when $(a,x)\in[a_1,a_2]\times[x_1,x_2]$, we can bound terms such as $\P(X_1\ge x)$ and $\P(X_1+X_2\ge x/2)$ by evaluating them at the points $(a_1,x_2)$ (for a lower bound) and $(a_2,x_1)$ (for an upper bound). We can bound the right-hand side of~\eqref{e:1} in a similar fashion as this expression is increasing in $x$ and decreasing in $a$.

We first deal with the case when $x$ is large: let us suppose that $x \ge x_0 = 100$. Recall that $a\in[0,1]$ and $\log(1+z)\le z$, so
\begin{align*}
 \frac{\P(X_1+X_2\ge x/2)}{\P(X_1\ge x)}
 &=\frac{8(a+x)^2((4a+x)^2+6(a+1)(x-4))}{(4a+x)^3(2a+x-2)}\\
 &\qquad+\frac{192(a+1)^2(a+x)^2}{(4a+x)^4}\log\left(1+\frac{x-2}{2a+2}\right)\\
 &\le \frac{8((4+x)^2+12(x-4))}{x(x-2)}+\frac{192}{x^4}(x-2)\\
 &\le 8+\frac{176}{x-2}+\frac{192}{x^3}\le 10
\end{align*}
for $x\ge x_0$. On the other hand,
\[
 \frac{2(x-1)\Lambda}{(a+1)(a+x)}\ge \frac{(x-1)\Lambda}{(x+1)}\ge 10
\]
for $x\ge x_0$ and any $\Lambda\ge 11$. Hence, it is clear that~\eqref{e:1} holds (with any $\delta \le 1$) when $x\ge x_0$.

We may now verify~\eqref{e:1} on $[0,1]\times[4,100]$ using a computer. We do this by establishing~\eqref{e:1} on a finite collection of subrectangles that partition $[0,1]\times[4,100]$; our program that does this proceeds as follows. On any rectangle within $[0,1]\times[4,100]$, we can bound both sides of~\eqref{e:1} as described above: if the required inequality on the whole rectangle follows from these bounds, our program certifies that the inequality holds on this rectangle, and if the required inequality cannot be shown to hold on the whole rectangle using the bounding strategy described above, then our program recursively divides such a rectangle into two smaller rectangles, and checks both of these subrectangles using the same strategy.

Our program verified~\eqref{e:1} for all $(a,x)\in[0,1]\times[4,100]$ with $\Lambda = 13.06207 < 14$; we set $\delta=10^{-10}$ to (very generously) allow for floating point inaccuracies and to provide the uniform bound needed in~\eqref{e:1}.
\end{proof}

We are now ready to prove our main result.

\begin{proof}[Proof of Theorem~\ref{t:main}]
Recall that we would like to show that if we have $\lambda > \Lambda = 13.06207$, $a \in [0,1)$, and probability distributions $\P_R$ and $\P_B$ such that $\P_B \succcurlyeq \Scr{F}(\lambda)$ and $\Scr{G}(a) \succcurlyeq \P_R$, then the outcome of a complete sequence of recolourings applied to $\Delta = \Delta(\P_R, \P_B)$ is almost surely a blue-win.

Recall that $\Delta$ is constructed as follows: take two i.i.d.\ sequences $(\CR_i)_{i \in \Z}$ and $(\CB_i)_{i\in\Z}$ of random variables with distributions $\P_R$ and $\P_B$ respectively, and then let $\Delta$ be the colouring of the real line given by
\[\dots +R_{-1}+B_{-1}+R_0+B_0+R_1+B_1+\dots,\]
where $R_i$ is a red interval with $|R_i|= \CR_i$ and $B_i$ is a blue interval with $|B_i| = \CB_i$ for each $i\in \Z$, and the origin is the boundary-point between $R_0$ and $B_0$. We shall construct, for each $t \ge 0$, a colouring $\Delta_t$ of the line into intervals and a representation of the colouring as
\[ \dots+C^{(t)}_{-1}+B^{(t)}_{-1}+C^{(t)}_0+B^{(t)}_0+C^{(t)}_1+B^{(t)}_1+\dots, \]
where  $C^{(t)}_i$ is a red-ended interval and $B^{(t)}_i$ is a blue interval for each $i \in \Z$. We shall also maintain, for each $t \ge 0$, a collection of bounds $\ell_t(.)$ such that $\ell_t(C^{(t)}_i)\ge r(C^{(t)}_i)$ and $\ell_t(B^{(t)}_i)\le |B^{(t)}_i|$ for each $i \in \Z$.

We start by taking $\Delta_0 = \Delta$, with $C^{(0)}_i = R_i$ and $B^{(0)}_i = B_i$. We then define our initial sequence of bounds as follows. Couple the lengths of the red intervals with an i.i.d.\ sequence $(\ell_0(R_i))_{i \in \Z}$ of $\Scr{G}(a)$ random variables so that $\ell_0(R_i)\ge |R_i| = r(R_i)$ for each $i \in \Z$, and couple the lengths of the blue intervals with an i.i.d.\ sequence $(\ell_0(B_i))_{i \in \Z}$ of $\Scr{F}(\lambda)$ random variables so that $\ell_0(B_i)\le |B_i|$ for each $i \in \Z$. Note also that this coupling can be done independently for each interval.

We get $\Delta_{t+1}$ from $\Delta_t$ by first coalescing some intervals and then updating the bounds for these intervals using the $\ell$-bounding argument described in the previous section, and subsequently rescaling the lengths of all the intervals in the resulting colouring. Since the initial bounds $(\ell_0(R_i))_{i \in \Z}$ and $(\ell_0(B_i))_{i \in\Z}$ are both i.i.d.\ sequences of random variables, it follows by induction that for each $t \ge 0$, the bounds $(\ell_t(C^{(t)}_i))_{i\in\Z}$ and $(\ell_t(B^{(t)}_i))_{i\in\Z}$ are also both i.i.d.\ sequences of random variables. We shall track the distributions of the bounds $\ell_t(.)$ to show that blue wins. In particular, we shall use the fact that the $\ell$-bounding argument is stable under stochastic domination to show that there are a pair of sequences $(\lambda_t)_{t \ge 0}$ and $(a_t)_{t \ge 0}$ of positive reals, with $\lambda_0 = \lambda$ and $a_0 = a$, such that the distributions of the i.i.d.\ sequences  $(\ell_t(C^{(t)}_i))_{i\in\Z}$ and $(\ell_t(B^{(t)}_i))_{i\in\Z}$   are $\Scr{G}(a_t)$ and  $\Scr{F}(\lambda_t)$ respectively. We will see that the sequence $(\lambda_t)_{t \ge 0}$ grows exponentially, whereas each $a_t$ is bounded away from $1$; the result will follow from this fact.

Since $a<1$ and $\lambda > \Lambda$, we may choose $\delta > 0$ such that $a \le 1-\delta$ and $\lambda \ge \Lambda/(1-\delta)$. Next, we fix an $\eps > 0$ such that $\eps < \min{(\delta/2, \eps_0, 1/10)}$, where $\eps_0$ is as in the statement of Lemma~\ref{l:reddom}. Since $2\eps < \delta \le 1$, it is easily verified that
\begin{equation}\label{e:epsdel}
\frac{(2-\delta + 2\eps)^2}{(2-\delta + \eps)^2(1+\eps)} \ge \frac{4}{(2-\eps)^2(1+\eps)} \ge 1 + \frac{\eps^2}{2}.
\end{equation}

For each $t \ge 0$, we construct $\Delta_{t+1}$ from $\Delta_t$ by first using the $\ell$-bounding argument with threshold $1 + \eps$, and then scaling down the lengths of all the intervals by a factor of $1+\eps$.

More precisely, given $\Delta_t$, we construct $\Delta_{t+1}$ by performing the following sequence of recolourings. First coalesce all the blue intervals $B$ with $\ell_t(B)\le 1+\eps$ with the surrounding red-ended intervals. After this recolouring, we see (by induction) that the estimates for the lengths of the blue intervals are i.i.d.\ random variables with distribution $1+\Exp(\lambda_t)$ conditioned on being at least $1+\eps$. Thus, they are distributed as $1+\eps+\Exp(\lambda_t)$ since conditioning an exponential random variable to be larger than some constant is equivalent to adding that constant. Let
\[
 \zeta_t=1-\exp{(-\eps/\lambda_t)}<\frac{\eps}{\lambda_t}
\]
be the probability that a blue interval is recoloured. Then each new red-ended interval $C$ is formed from $Y$ consecutive red-ended intervals $C_1, C_2, \dots,C_Y$, where $Y$ has distribution $\Geom(\zeta_t)$, independently for each resulting red-ended interval $C$. We combine the length bounds as in the $\ell$-bounding argument using Corollary~\ref{c:multi} and set 
\[
 \ell_{t+1}(C)=2^{\lceil\log_2Y\rceil}\sum_{i=1}^Y\ell_t(C_{i}).
\]
By Lemma~\ref{l:reddom}, this is stochastically bounded above by a $\Scr{G}(a_t+\Lambda\zeta_t)$ distribution. By increasing the  values of these bounds if necessary, we can assume that these bounds are now i.i.d.\ random variables with distribution $\Scr{G}(a_t+\Lambda\zeta_t)$.

Each surviving blue interval now has length strictly greater than $1+\eps$. We now recolour those red-ended intervals $C$ with $\ell_{t+1}(C)\le 1+\eps$. The probability that this occurs for a given red-ended interval is
\[
 \xi_t=1-\frac{(a_t+\Lambda\zeta_t+1)^2}{(a_t+\Lambda\zeta_t+1+\eps)^2}
 =\frac{\eps(2a_t+2\Lambda\zeta_t+2+\eps)}{(a_t+\Lambda\zeta_t+1+\eps)^2}.
\]
Each new blue interval $B$ is formed from $Y$ consecutive blue intervals $B_1, B_2, \dots,B_Y$, where $Y$ has distribution $ \Geom(\xi_t)$, independently for each resulting blue interval. We set
\[
\ell_{t+1}(B) = \sum_{i=1}^Y\ell_t(B_i)
\] 
as in the $\ell$-bounding argument. As the sum of $\Geom(\xi_t)$ independent $\Exp(\lambda)$ random variables is an $\Exp(\lambda/(1-\xi_t))$ random variable, the resulting bounds for the lengths of the blue intervals stochastically dominate a $1+\eps+\Exp(\lambda_t/(1-\xi_t))$ random variable; by decreasing the values of these bounds if necessary, we can assume that these bounds are now i.i.d.\ random variables with distribution $1+\eps+\Exp(\lambda_t/(1-\xi_t))$. Finally, note that the estimates for the red-content of the surviving red-ended intervals are now distributed according to the distribution of a random variable with distribution $\Scr{G}(a_t+\Lambda\zeta_t)$ that is conditioned on being at least $1+\eps$. This is easily seen to be $1+\eps$ times a random variable with distribution $\Scr{G}((a_t+\Lambda\zeta_t)/(1+\eps))$. 

We now get $\Delta_{t+1}$ by scaling down the lengths of all the intervals (and our estimates for these) by a factor of $1+\eps$. Of course, we have not specified how the red-ended intervals and blue intervals are indexed in $\Delta_{t+1}$. It is however clear (by induction) that the law of $\Delta_{t+1}$ is shift-invariant, so the precise choice of origin clearly does not affect the outcome of the evolution; therefore, we arbitrarily choose a red-ended interval $C$ and a blue interval $B$ in $\Delta_{t+1}$ such that $C$ is immediately to the left of $B$, and then designate $C$ and $B$ to be $C^{(t+1)}_0$ and $B^{(t+1)}_0$ respectively.

We now know that the i.i.d.\ sequences  $(\ell_t(C^{(t+1)}_i))_{i\in\Z}$ and $(\ell_t(B^{(t+1)}_i))_{i\in\Z}$   have distributions $\Scr{G}(a_{t+1})$ and $\Scr{F}(\lambda_{t+1})$ respectively, where
\[
a_{t+1}= \frac{a_t+\Lambda\zeta_t}{1+\eps},
\]
and
\[
\lambda_{t+1}=\frac{\lambda_t}{(1-\xi_t)(1+\eps)}.
\]
We shall show by induction that $a_t \le 1-\delta$ and $\lambda_t \ge (1+\eps^2/2)^t \lambda$  for each $t\ge 0$. Indeed, inductively assume that this is true of $a_t$ and $\lambda_t$. Since $\zeta_t < \eps/\lambda_t \le \eps / \lambda$ and $\Lambda/\lambda \le 1-\delta$, it follows that
\[
 a_{t+1}= \frac{a_t+\Lambda\zeta_t}{1+\eps} \le \frac{a_t+\Lambda\eps/\lambda}{1+\eps} \le \frac{1-\delta + (1-\delta)\eps}{1+\eps} = 1-\delta.
\]
Next, using~\eqref{e:epsdel}, we note that
\begin{align*}
\frac{\lambda_{t+1}}{\lambda_t} = \frac{1}{(1-\xi_t)(1+\eps)} &= \frac{(a_t + \Lambda\eps/\lambda_t + 1 + \eps)^2}{(a_t + \Lambda\eps/\lambda_t + 1)^2(1+\eps)}\\
&\ge \frac{(2-\delta + 2\eps)^2}{(2-\delta + \eps)^2(1+\eps)}\ge 1 + \frac{\eps^2}{2}.
\end{align*}
Thus, the sequence $(\lambda_t)_{t \ge 0}$ grows exponentially, whereas the sequence $(a_t)_{t \ge 0}$ is bounded away from 1.

We now finish the proof as in the proof of Theorem~\ref{t:renorm} by showing that there is a positive probability that there exists a point which only changes colour a finite number of times and is ultimately blue; the existence of any such point clearly precludes a red-win or a tie, and therefore implies that blue wins almost surely. 

We say that a point is blue \emph{externally} in $\Delta_t$ if the point is blue, but not contained within one of the red-ended intervals of $\Delta_t$. Since the law of $\Delta_t$ is shift-invariant for each $t \ge 0$,  the probability that an externally blue point of $\Delta_t$ is no longer externally blue in $\Delta_{t+1}$ is the same for each externally blue point of $\Delta_t$; furthermore, this common probability is at most $\zeta_t <\eps/\lambda_t$. It follows that the probability that any given externally blue point of $\Delta_T$ remains externally blue in each $\Delta_{t}$ with $t > T$ is at least $1 - \sum_{i > T} \eps/\lambda_i$. Now, choose $T \ge 0$ to be large enough so that $\sum_{i > T} \eps/\lambda_i < 1$; this is possible since the sequence $(\lambda_t)_{t \ge 0}$ grows exponentially. It is now clear that any given externally blue point of $\Delta_T$ is never recoloured again with positive probability; it follows that blue wins almost surely, as required.
\end{proof}

\section{A comparison of two simple distributions}\label{s:application}
In this section, we consider as an example the case where the initial lengths of all the red intervals are deterministically $1$ and those of the blue intervals are distributed according to the (uniform) $U[0,1+\gamma]$ distribution with $\gamma \ge 0$.

Before we turn to the proof of Theorem~\ref{t:toy}, we make the following simple observation.

\begin{proposition}\label{c:boundbygeom}
Let $p>0$ and $0 < a <b$. If $X$ and $Y$ are random variables with distributions $\Geom(p)\circ U[a,b]$ and $\Geom(p)$ respectively, then
\[
 \P(X\ge x)\ge\frac12\P\left(Y\ge \frac{2x}{a+b}\right).
\]
\end{proposition}
\begin{proof}
To see this, write $X$ as a sum of $Y$ independent copies of a random variable with distribution $U[a,b]$; this sum is at least $(a+b)Y/2$ with probability $1/2$, by symmetry.
\end{proof}

\begin{proof}[Proof of Theorem~\ref{t:toy}]
Let $\Delta_\gamma$ denote a random colouring of the line into intervals where the lengths of all the red intervals are deterministically $1$ and those of the blue intervals are distributed according to the (uniform) $U[0,1+\gamma]$ distribution. 
We wish to show that the outcome of linear coalescence starting from $\Delta_\gamma$ is almost surely a red-win if $\gamma$ is suitably small, and almost surely a blue-win if $\gamma$ is sufficiently large.

For $p\in[0,1]$, we write $\Delta_{p, \gamma}$ for the colouring of the line into intervals obtained from $\Delta_\gamma$ as follows: we first recolour all blue intervals of length less than $1$ (which we may do since all red intervals initially have length $1$), and having recoloured all such blue intervals, we then recolour, independently with probability $p$, those red intervals which (still) have length exactly 1 (which we may again do since we all blue intervals now have length greater than $1$).

After recolouring all the blue intervals of length less than $1$ (which recolours a $1/(1+\gamma)$ proportion of all the blue intervals), it is easy to check that the distribution of the lengths of the blue intervals is $1+U[0,\gamma]$, and that of the red intervals is 
\[1+\left(-1+\Geom\left(\frac{1}{1+\gamma}\right)\right)\circ (1+U[0,1]).\]
Next, after we recolour, independently with probability $p$, those red intervals of length exactly $1$ (which recolours a $p\gamma/(1+\gamma)$ proportion of all the red intervals), we have, after expanding out the geometric distributions for clarity, writing the distributions of both colours as a mixture of a sequence of simpler distributions, and setting $q = 1-p$, the following distributions of lengths.

\[\begin{array}{cc|cc}
\text{Red}&\text{Probability}&\text{Blue}&\text{Probability}\\\hline
1& \frac{q\gamma}{q\gamma+1}&1+1\circ U[0,\gamma]&\frac{q\gamma+1}{1+\gamma}\\
2+1\circ U[0,1]&\frac{\gamma}{(q\gamma+1)(1+\gamma)}&
3+2\circ U[0,\gamma]&(\frac{q\gamma+1}{1+\gamma})(\frac{p\gamma}{1+\gamma})\\
3+2\circ U[0,1]&\frac{\gamma}{(q\gamma+1)(1+\gamma)^2}&
5+3\circ U[0,\gamma]&(\frac{q\gamma+1}{1+\gamma})(\frac{p\gamma}{1+\gamma})^2\\
\dots&\dots&\dots&\dots\\
\end{array}\]

More precisely, the lengths of distinct intervals in $\Delta_{p, \gamma}$ are independent, and the distributions of the lengths of the red and blue intervals are as follows. The length of a red interval, with probability $q\gamma/(q\gamma+1)$, is exactly $1$, and with probability $\gamma/((q\gamma+1)(1+\gamma)^k)$, is drawn from the distribution $(k+1)+k\circ U[0,1]$ for each $k \ge 1$. This claim perhaps merits a few words of explanation. After the second round of recolourings, a $(1-p)\gamma/(1+\gamma) = q\gamma/(1+\gamma)$ fraction of the red intervals from after the first round of recolourings have length exactly $1$; noting that a $p\gamma/(1+\gamma)$ fraction of the red intervals from after the first round are recoloured (and consequently removed) after the second round, we arrive at the probability that a red interval has length exactly $1$ after the second round,  which is $(1 - p\gamma/(1+\gamma))^{-1}(q\gamma/(1+\gamma)) = q\gamma / (q\gamma + 1)$. We may reason analogously in the other cases as well. Similarly, it may be seen that the length of a blue interval is drawn from the distribution $(2k-1)+k\circ U[0,\gamma]$ with probability $((q\gamma+1)/(1+\gamma))(p\gamma/(1+\gamma))^{k-1}$ for each $k \ge 1$. Let us write $\Scr{R}_p(\gamma)$ and $\Scr{B}_p(\gamma)$ respectively to denote these distributions of the lengths of the red and blue intervals in $\Delta_{p, \gamma}$.

To show that a particular colour wins, we shall choose a suitable value of $p \in [0,1]$ and apply Theorem~\ref{t:main} appropriately to the distributions of the lengths of the red and blue intervals of $\Delta_{p,\gamma}$.

\textbf{Case 1: $\gamma < \gamma_R = 0.1216$.} We wish to show that red wins in this case. We fix $p = 1$ and $q = 0$ and write $\Scr{R}(\gamma)$ and $\Scr{B}(\gamma)$ for the distributions $\Scr{R}_{1}(\gamma)$ and $\Scr{B}_{1}(\gamma)$ respectively. It is sufficient to show that $\Scr{R}(\gamma) \succcurlyeq \Scr{F}(\Lambda)$ and $\Scr{G}(a) \succcurlyeq \Scr{B}(\gamma) $ for some $a<1$, where $\Lambda$ is as in the statement of Theorem~\ref{t:main}. 

We first show that when $\gamma$ is sufficiently small, $\Scr{G}(a) \succcurlyeq \Scr{B}(\gamma)$ for some $a<1$. Suppose that $0 \le \gamma \le \gamma_R$, and let $X$ and $Y$ be random variables with distributions $\Scr{B}(\gamma)$ and $\Scr{G}(a)$ respectively. Since both $\Scr{B}(\gamma)$ and $\Scr{G}(a)$ are supported on $[1,\infty)$, we need to show for each $x \ge 1$ that 
\[ \P(X \ge x) \le \P (Y \ge x).\]

First, when $x \in [1,1+\gamma]$, we note that the density function of $\Scr{B}(\gamma)$ in this range is $1/(\gamma(1+\gamma)) > 7$, while the density function of $\Scr{G}(a)$ in this range is $2(a+1)^2/(a+x)^3\le 2/(a+1)\le 2$. Thus, the stochastic domination condition holds when $x \in [1, 1+\gamma]$. Now assume that $x\in[2k-1+k\gamma,2k+1+(k+1)\gamma]$ for some $k \ge 1$. In this case, we note that
\[ 
\P(X\ge x) \le \P(X \ge 2k-1+k\gamma) \le  \left(\frac{\gamma}{1+\gamma}\right)^k\le 9^{-k}
\]
and then check, for $a$ sufficiently close to $1$, that we have
\[
\P(Y\ge x) \ge \frac{(a+1)^2}{(a+2k+1+(k+1)\gamma)^2} 
\ge \frac{3.99}{((k+1)(2+\gamma))^{2}}\ge \frac{3.99}{5(k+1)^2}.
\]
It is then easy to see that $3.99/(5(k+1)^2)\ge 9^{-k}$ for all $k\ge1$.

We now wish to show that $\Scr{R}(\gamma) \succcurlyeq \Scr{F}(\Lambda)$ for all $0 \le \gamma \le \gamma_R$. The distributions $\Scr{R}(\gamma)$ stochastically decrease with $\gamma$; indeed, this follows from the fact that the $\Geom(1/(1+\gamma))$ distributions stochastically decrease with $\gamma$. Hence, it suffices to show that $\Scr{R}(\gamma_R) \succcurlyeq \Scr{F}(\Lambda)$.

Now, let $X$ and $Y$ have distributions $\Scr{R}(\gamma_R)$ and $\Scr{F}(\Lambda) = $1$ + \Exp(\Lambda)$ respectively. Since both distributions are supported on $[1, \infty)$, we need to show for each $x \ge 0$ that
\begin{equation}\label{e:dom_1}
\exp{\left(-\frac{x}{\Lambda}\right)} = \P(Y \ge x+1) \le \P (X \ge x+1).
\end{equation}
Applying Proposition~\ref{c:boundbygeom}, we have for each $x \ge 0$, 
\begin{equation}\label{e:bbnd}
\P(X\ge x+1)\ge \frac12(1+\gamma_R)^{-2x/3}.
\end{equation}
We need this to be at least $e^{-x/\Lambda}$; unfortunately this does not hold for small $x$. However, note that this holds for $x\ge x_0=2000$ since
\[
0.1216 = \gamma_R < e^{(3/2\Lambda)-(3\log 2)/(2x_0)}-1
\]
which establishes~\eqref{e:dom_1} on $[x_0, \infty)$. Then we inductively define $x_i$ by 
\begin{equation}\label{e:seq}
  e^{-x_{i+1}/\Lambda}=\P(X\ge x_i+1)
\end{equation}
and check with the help of a computer that there exists an $n\in \N$ such that $x_n = 0$ (i.e., $x_{n-1}<1$), and that the sequence $(x_i)_{i=0}^n$ is monotone decreasing. Consequently, we see that~\eqref{e:dom_1} holds on $[x_{i+1},x_i]$ by~\eqref{e:seq} for each $i\ge 0$. We briefly sketch how we check this claim on a computer. We can bound the distribution $\Scr{R}(\gamma_R)$ numerically in terms of sums of uniform distributions by expanding out the geometric distribution in its definition. For large $k$, the distribution $k\circ U[a,b]$ can be bounded by the Berry--Esseen theorem in the form proved in~\cite{BE}. This gives a bound of $0.5751/\sqrt{k}$ between the cumulative distribution functions of $k\circ U[a,b]$ and the corresponding normal approximation. The effect of this approximation is that it `almost' removes the factor of $1/2$ in~\eqref{e:bbnd}, allowing us to prove stochastic domination down to much smaller $x$. This is still inadequate for very small $x$ however, so we finish by calculating, for small $k$, the (Irwin--Hall) distribution of $k\circ U[a,b]$ exactly in terms of piecewise polynomial functions.

\textbf{Case 2: $\gamma \ge \gamma_B = 6.048$.} In this case, we would like to show that blue wins. Choose a constant $c > 5/4$ and set $p=1-c/\gamma$ so that $q\gamma = c$.

Our aim is to show that if $\gamma$ is sufficiently large, then $\Scr{G}(a) \succcurlyeq \Scr{R}_{p}(\gamma)=\Scr{R}(\gamma)$ for some $a<1$ and $ \Scr{B}_{p}(\gamma)=\Scr{B}(\gamma)\succcurlyeq \Scr{F}(\Lambda)$. 

We first show that the distribution of the lengths of the red intervals of $\Delta_{p,\gamma}$ is stochastically dominated by a $\Scr{G}(a)$ distribution with $a<1$. Let $X$ and $Y$ be random variables with distributions $\Scr{R}(\gamma)$ and $\Scr{G}(a)$ respectively.
For $x\in[1,2]$ it is enough to show that
\[
 \P(X>x) \le \P(X>1) = 1-\frac{q\gamma}{q\gamma+1}
 \le \frac{(a+1)^2}{(a+2)^2}=\P(Y\ge2) \le \P(Y>x)
\]
which, since $q\gamma=c>5/4$, is true when $a$ is sufficiently close to 1. For $x\in[2,3]$, we note that the density function of $\Scr{R}(\gamma)$ is $\gamma/((q\gamma+1)(1+\gamma))$, which is pointwise greater than the corresponding density function $2(a+1)^2/(a+x)^3$ of $\Scr{G}(a)$ provided $a$ is sufficiently close to 1. Now, for $x\in[2k+1,2k+3]$ with $k\ge1$ and $a$ sufficiently close to $1$, we have
\begin{align*}
 \P(X\ge x)&\le \frac{1}{(q\gamma+1)(1+\gamma)^k}\le 6^{-k}/2 \text{, and}\\
 \P(Y \ge x)&\ge \frac{(a+1)^2}{(a+2k+3)^2}\ge \frac{9(k+2)^{-2}}{10};
\end{align*}
it is clear that $9(k+2)^{-2}/10\ge 6^{-k}/2$ for each $k\ge 1$.

Finally we show that $\Scr{B}(\gamma) \succcurlyeq \Scr{F}(\Lambda)$ for all sufficiently large $\gamma$ using a strategy similar to the one used in the previous case. 

Recall, the definition of $\Scr{B}(\gamma)$: a random variable with this distribution is drawn with probability $((q\gamma+1)/(1+\gamma))(p\gamma/(1+\gamma)^{k-1})$ from the distribution $(2k-1)+k\circ U[0,\gamma]$ for each $k \ge 1$. We claim that the distributions $\Scr{B}(\gamma)$ stochastically increase with $\gamma$. This follows from the fact that the distributions $U[0,\gamma]$ and $\Geom{((\gamma - c)/(1+\gamma))}$ are both stochastically increasing in $\gamma$.

Hence, it is sufficient to show that $\Scr{B}(\gamma_B) \succcurlyeq \Scr{F}(\Lambda)$. Let $X$ be a random variable with distribution $\Scr{B}(\gamma_B)$. We would like to show, for all $x \ge 0$, that 
\begin{equation}\label{e:dom_2}
\P(X\ge x+1) \ge e^{-x/\Lambda}.
\end{equation}

Let $Y$ be a (geometric) $\Geom(p\gamma_B/(1+\gamma_B)) = \Geom((\gamma_B - c)/(1+\gamma_B))$ random variable. We deduce from Proposition~\ref{c:boundbygeom} that
\begin{align*}
 \P(X\ge x+1)&\ge
 \frac12 \P(Y\ge (x+2)/(2+\gamma_B/2))\\
 &\ge \frac{1}{2}\left(\frac{\gamma_B-c}{1+\gamma_B}\right)^{(x+2)/(2+\gamma_B/2)}
\end{align*}
for $x\ge x_0=10^6$. This bound can be checked to be at least $e^{-x/\Lambda}$ for all $x \ge x_0$ using the fact that $\gamma_B=6.048$. This shows that~\eqref{e:dom_2} holds on $[x_0, \infty)$. As before, we inductively define $x_i$ by $e^{-x_{i+1}/\Lambda}=\P(X\ge x_i+1)$ and check with the help of a computer that there exists an $n\in \N$ such that $x_n <2$, and that the sequence $(x_i)_{i=0}^n$ is monotone decreasing. This verifies~\eqref{e:dom_2} on $[2, x_0]$. Finally, it is easy to check~\eqref{e:dom_2} on $[0,2]$ as the density function $((q\gamma_B + 1)/(1+\gamma_B))/\gamma_B$ of $\Scr{B}(\gamma_B)$ is pointwise less than the density function of $\Scr{F}(\Lambda)$ in this region since $e^{-x/\Lambda}/\Lambda > e^{-2/\Lambda}/\Lambda$.
\end{proof}

\section{Conclusion}\label{s:conc}
It is possible that there exist natural analytical conditions on the red and blue distributions that are useful in determining the outcome of linear coalescence. It would be very interesting to determine such sufficient conditions if they exist; in this paper, we have managed to rule out two very natural candidates. 

There likely exist quite `dissimilar' distributions $\P_R$ and $\P_B$ for which the outcome of the coalescence process is a tie. To prove such a result analytically, it would seem necessary to track the process fairly precisely. It is unclear whether the methods developed here would be of much help in such a task; new ideas are probably required. However, it is possible that there exists a natural topology on the space of probability distributions on the positive reals with respect to which the relation $\rhd$ is open; we would not be surprised if one could use such a result to demonstrate the existence of dissimilar distributions $\P_R$ and $\P_B$ for which linear coalescence results in a tie.

We remind the reader of one side-effect of the absence of monotonicity in linear coalescence. It would be tempting to prove a high confidence result which is stronger than Theorem~\ref{t:toy} in a manner analogous to the proof of Claim~\ref{t:counter}. Indeed, when comparing the red distribution that is deterministically $1$ with the $U[0,1+\gamma]$ blue distribution, we can show that with high confidence, red wins almost surely when $\gamma=1.16$, and blue wins almost surely when $\gamma=1.19$. Unfortunately we cannot deduce that red wins when $\gamma<1.16$ or that blue wins when $\gamma>1.19$ from this. This illustrates one key drawback of proving results with high confidence: such results can only be applied to specific pairs of distributions, and in the absence of monotonicity, we are unable to do much more than speculate.

Perhaps the most important question not addressed in this paper which merits investigation is that of devising analytical techniques to compute the probability that a large monochromatic central interval appears in the coalescence process on a large finite interval. In addition to being an interesting question in its own right, it would help transform the high confidence results in this note into theorems proper. 

\section*{Acknowledgements}
The first and second authors were partially supported by NSF grant DMS-1600742, and the second author also wishes to acknowledge support from EU MULTIPLEX grant 317532.

Some of the research in this paper was carried out while the third and fourth authors were visitors at the University of Memphis and was continued while the authors were visitors at the IMT Institute for Advanced Studies Lucca. The authors are grateful to Guido Caldarelli and the other members of the Complex Networks Group at IMT Lucca for their hospitality, and in addition, the third and the fourth authors are grateful for the hospitality of the University of Memphis.

Finally, we would like to thank the anonymous referees for their careful reading of this paper as well as their many helpful comments.

\bibliographystyle{amsplain}
\bibliography{coalescence}

\end{document}